\documentclass[12pt,a4paper]{article}

\usepackage[english]{babel}
\usepackage{amsmath,amssymb,amsfonts}
\usepackage[latin1]{inputenc}
\usepackage{graphicx}
\usepackage{amsthm}
\usepackage{color}
\usepackage{stmaryrd}
\usepackage{fullpage}
\usepackage{enumerate}

\title{An existence result for the steady rotating Prandtl equation}
\date{March 2016}
\author{Anne-Laure \textsc{Dalibard} \& Matthew \textsc{Paddick}\footnote{Sorbonne Universit\'es, UPMC Univ Paris 06, UMR 7598, Laboratoire Jacques-Louis Lions, F-75005, Paris, France}~\footnote{CNRS, UMR 7598, Laboratoire Jacques-Louis Lions, F-75005, Paris, France}}

\begin{document}

\renewcommand{\labelitemi}{$\bullet$}
\newtheorem{theo}{Theorem}[section]
\newtheorem{lemma}[theo]{Lemma}
\newtheorem{propo}[theo]{Proposition}
\newtheorem{coro}[theo]{Corollary}
\newtheorem*{rmk}{Remark}
\newtheorem*{defi}{Definition}
\newtheorem*{nota}{Notations}

\newcommand{\preu}[1]{\textit{\underline{#1}}}
\newcommand{\thref}[1]{Theorem \ref{#1}}
\newcommand{\lemref}[1]{Lemma \ref{#1}}
\newcommand{\propref}[1]{Proposition \ref{#1}}
\newcommand{\cororef}[1]{Corollary \ref{#1}}

\newcommand{\rplus}{\mathbb{R}^{+}}
\renewcommand{\div}{\mathrm{div}~}
\newcommand{\curl}{\mathrm{curl}~}
\newcommand{\eps}{\varepsilon}
\newcommand{\norme}[2]{\left\| #2 \right\| _{#1}}
\newcommand{\derp}[1]{\partial_{#1}}
\newcommand{\dOmega}{\partial\Omega}
\newcommand{\Lip}{\mathrm{Lip}}
\newcommand{\mathand}[1]{~~\mathrm{#1}~~}
\newcommand{\vphi}{\varphi}
\newcommand{\bb}[1]{\mathbb{#1}}
\newcommand{\Rr}{\mathbb{R}}
\newcommand{\Ac}{{\cal A}}
\newcommand{\Cc}{{\cal C}}
\newcommand{\Dc}{{\cal D}}
\newcommand{\Lc}{{\cal L}}
\newcommand{\Mc}{{\cal M}}
\newcommand{\Nc}{{\cal N_\tau}}
\newcommand{\Oc}{{\cal O}}
\newcommand{\gap}{\vspace{10pt}}
\newcommand{\be}{\begin{equation}}
\newcommand{\ee}{\end{equation}}
\newcommand{\ba}{\begin{aligned}}
\newcommand{\ea}{\end{aligned}}
\newcommand{\p}{\partial}
\newcommand{\psbl}{\psi^{BL}}
\newcommand{\wme}{\underline{w}^-_\eps}

\maketitle

\begin{abstract}
We consider a steady, geophysical 2D fluid in a domain, and focus on its western boundary layer, which is formally governed by a variant of the Prandtl equation.
 By using the von Mises change of variables, we show that this equation is well-posed under the assumption that the trace of the interior stream function has large variations, and that the variations in the coastline profile are moderate.
\gap

\end{abstract}

\section{Introduction}


The goal of this article is to prove the global existence of solutions of the stationary Prandtl-like equation
\begin{equation} \displaystyle \left\{ \begin{array}{rcl} \lambda_0\nu(y)(u\derp{\xi}v+v\derp{y}v) + \psi - \nu(y)^2\derp{\xi}^2 v & = & \psi^0(y) \\
(u,v) & = & \nabla^\bot \psi := (-\derp{y}\psi,\derp{\xi}\psi) \\
v|_{y=0} & = & v_0(\xi) \\
(u,v)|_{\xi=0} & = & 0 \\
\lim_{\xi\rightarrow +\infty} \psi(\xi,y) & = & \psi^0(y),  \end{array} \right. \label{pr} \end{equation}
in the domain $\Omega = \{(\xi,y)\in\Rr^2~|~y>0, \ \xi>0\}$. The function $\psi^0$ is given and smooth, $\lambda_0$ is a strictly positive parameter, and the function $\nu$ is a smooth function with $\nu\geq 1$.

This equation arises in geophysical models to describe the behaviour of western boundary oceanic currents in certain regimes.
 We describe the physical assumptions and scaling leading to \eqref{pr} in paragraph \ref{ssec:physique} after the statement of our main result, but let us merely mention that the role of the function $\nu$ is to take into account the geometry of the western coast.
  
Note that equation \eqref{pr} is similar to the stationary Prandtl equation, which is\footnote{Note that here $y$ is the tangential variable and $\xi$ is the rescaled normal variable,
 in contrast with the usual convention in the study of the Prandtl equation.
 We have made this choice to stick to the geophysical setting, in which $u$ is the East-West component of the velocity and $v$ its North-South component.}
\be\label{pr-usual}
 \displaystyle \left\{ \begin{array}{rcl} u\derp{\xi}v+v\derp{y}v  - \derp{\xi}^2 v & = & -\frac{dp_E}{dy} \\
(u,v) & = & \nabla^\bot \psi := (-\derp{y}\psi,\derp{\xi}\psi) \\
v|_{y=0} & = & v_0(\xi) \\
(u,v)|_{\xi=0} & = & 0 \\
\lim_{\xi\rightarrow +\infty} v(\xi,y) & = & v_E(y),  \end{array} \right.
\ee
where the functions $v_E$ and $p_E$ are given (they are the trace of some outer Euler flow) and satisfy $v_E v_E'=-\frac{dp_E}{dy} $. Notice that here we have $\nu(y)=1$, which corresponds to a flat boundary.
 The main differences between the usual Prandtl equation \eqref{pr-usual} and equation \eqref{pr} lie in the presence of the additional term $\psi$ in the equation
 (which is due to rotation, as we will explain in paragraph \ref{ssec:physique}) and in the condition at infinity, which bears on the stream function $\psi$ rather than on the velocity $v$.
 These two differences will result in significant changes in the analysis of equation \eqref{pr}.

The mathematical analysis of the stationary Prandtl equation \eqref{pr-usual} goes back to the seminal work of Ole\u{i}nik \cite{Oo63,OSbook},
 who proved that if $v_0$ is smooth such that $v_0\geq 0$, $v_0'(0)>0$ and $v_0(\xi)>0$ for $\xi>0$, then solutions of \eqref{pr-usual} exist, at least locally in $y$ and globally if $\frac{dp_E}{dy}\leq 0$.
 Her idea was to consider \eqref{pr-usual} as a non-local and non-linear evolution equation, with the variable $y$ playing the role of time.
 With this point of view, as long as $v$ remains non-negative, the equation has a parabolic structure, and is therefore locally well-posed.
 The method of proof relies on the use of a nonlinear change of variables due to von Mises, which transforms \eqref{pr-usual} into a local non-linear diffusion equation.

The aim of this paper is to adapt these ideas to equation \eqref{pr}, and to exhibit mathematical conditions on $\lambda_0, \nu, \psi^0$  that would ensure
 the existence and uniqueness of a solution such that $v$ remains positive. In the course of the proof, we will focus on the differences between equations \eqref{pr} and \eqref{pr-usual}.
 We will also discuss the physical meaning of our conditions.

\subsection{Short review of previous works on equation \eqref{pr}}

Some elements of analysis are given in the book by Pedlosky \cite{Pjbook}, following the analysis of Ierley and Ruehr \cite{IerleyRuehr}.
 In these works, it is assumed that $\nu(y)\equiv 1$ and that $\psi^0(y)=\alpha y+1$ for some $\alpha>0$ and that $\chi\equiv 0$ in order to simplify matters.
 In this case, it is natural to seek a solution in which the variables are separated, that is $\psi(\xi,y) = \phi(\xi)\psi^0(y)$ (see \cite{Pjbook}). Plugging this into (\ref{pr}), we get the ODE on $\phi$
\begin{equation} \phi''' = \lambda_0\alpha((\phi')^2-\phi\phi'')+\phi-1 \label{ode} . \end{equation}
The equation has an equilibrium at $\phi\equiv 1$, and we now examine the  linearised equation around this equilibrium, which reads
\be\label{linearized}
\varphi^{(3)}=- \lambda_0  \alpha \varphi'' + \varphi.
\ee
The solutions of this equation are of the form $\varphi=C\exp(-a\xi)$, with $a$ satisfying
\begin{equation} p_\alpha(a) := \lambda_0\alpha a^2 - a^3 - 1=0 \label{fpn} . \end{equation}
A quick study of the function $a\in \Rr\mapsto p_\alpha(a)$ shows that $p_\alpha$ always has exactly one negative root, and two strictly positive roots if and only if 
\begin{equation} \lambda_0\alpha > \sqrt[3]{\frac{27}{4}} \label{psizero} . \end{equation}
If $\lambda_0\alpha<\sqrt[3]{\frac{27}{4}}$, then $p_\alpha$ has two complex conjugate roots with positive real part and non-zero imaginary part. 
In the case where $\psi^0$ is affine, we therefore expect solutions of \eqref{ode} to behave as $\xi\to \infty$ as linear combinations of $\exp(-a_\pm \xi)$, where $a_{\pm}$ are the two roots of $p_\alpha$ with positive real part
 (putting aside the degenerate case $\lambda_0\alpha= \sqrt[3]{\frac{27}{4}}$).

At this stage, let us recall that we look for solutions such that $v>0$, or in other words, such that $\psi$ is strictly increasing with respect to $\xi$.
 This choice is made necessary by the Prandtl-like (or diffusion-like) structure of equation \eqref{pr}: indeed, the theoretical construction of solutions of \eqref{pr-usual}
 when $v$ takes negative values is widely open, even though numerical schemes relying on the theory of ``interactive boundary layers'' or ``triple deck solutions'' exist (see, for example, P-Y. Lagr\'ee \cite{Lpy}).
 Therefore we want to consider solutions of \eqref{ode} such that $\phi$ is increasing in $\xi$. The linearised equation suggests that a necessary condition is \eqref{psizero}:
 indeed, if $\lambda_0$ is small, then we expect the solutions in a neighbourhood of $\phi=1$ to behave like
$$
1 - C\cos (\Im(a_+)\xi) \exp(-\Re(a_+)\xi),
$$
and therefore to be non-monotonous.

This heuristic analysis is confirmed by the study of the case where $\lambda_0=0$ (linear) on one hand, and by the numerical simulations and analytic computations of Ierley and Ruehr \cite{IerleyRuehr}
 on the other. Indeed, in the linear case, the solution of \eqref{ode} endowed with the conditions $\phi(0)=\phi'(0)=0$ is exactly
$$
\phi(\xi)= 1 + \frac{4}{\sqrt{3}}\cos \left(\frac{\sqrt{3}}{2} \xi + \frac{\pi}{6} \right) \exp\left(-\frac{\xi}{2}\right).
$$ 
Hence it is easily checked that $\phi$ changes sign. In the case when $\lambda_0$ is small but non zero, the numerical simulations  and analytic computations of \cite{IerleyRuehr},
 confirmed by the later study of Smith and Mallier \cite{SmithMallier}, suggest that there exists a critical value of $\lambda$ under which $v$ may change sign, and above which $v$ remains positive.
 The purpose of this article is to give some theoretical foundation to these observations: we will indeed prove that if $\lambda_0$ is large enough, under suitable assumptions on $\psi^0$,
 then solutions of \eqref{pr} exist with $v>0$.

\subsection{Main result}

Throughout the article, the polynomial
$$
P_y(t)= \nu(y)^2 t^3 - \nu(y)\lambda_0 (\psi^0)'(y) t^2 +1,\quad t\in \Rr
$$
will play a fundamental role. Studying its variations closely, it can be checked that if $\lambda_0 (\psi^0)'(y)/\nu(y)^{1/3}>\sqrt[3]{\frac{27}{4}}$, $P_y$ has exactly two positive roots.
 We denote $a(y)$ the smallest of these roots, so that $P_y$ is decreasing in the vicinity of $a(y)$.

The existence result is the following.

\begin{theo} Let $\psi^0\in \mathcal C^2([0,Y])$ such that $\inf\psi^0>0,~ \inf \psi_0'>0$.  
 Assume that $\lambda_0 (\psi^0)'(y)/\nu(y)^{1/3} > 2+\eta>\sqrt[3]{\frac{27}{4}}$ for every $y\geq 0$ and for some $\eta>0$.
 Let $v_0 \in W^{2,\infty}(\rplus)$, positive on $]0,+\infty[$, with $v(0)=0$, $v_0'(0)>0$,
 such that $\int_0^{+\infty} v_0(\xi)~d\xi = \psi^0(0)$. We assume that $v_0$ satisfies the corner compatibility condition
 \begin{equation} \nu(0)^2   v_0''(\xi) + \psi^0(0) = \Oc (\xi^2)  \label{comppr} \end{equation}
 as $\xi \rightarrow 0$, and the decay constraint
 \be \label{comp-infty} v_0(\xi) \sim a(0) (\psi^0(0)-\psi_0(\xi)) \ee
 as $\xi\rightarrow +\infty$, in which $\psi_0(\xi)=\int_0^\xi v_0(s)~ds$ is the initial stream function.

 Then, for every $Y>0$, there exists a unique classical Lipschitz solution to the steady rotating Prandtl equation
 \begin{equation} \displaystyle \left\{ \begin{array}{rcl} \lambda_0\nu(y)\left[ u \p_{\xi } v + v \p_{y} v  \right] + \psbl - \nu(y)^2\p_\xi^2 v & = & \psi^0(y) \\
(u,v) & = & \nabla^\bot \psi \\
v|_{y=0} & = & v_0(\xi) \\
(u,v)|_{\xi=0} & = & 0 \\
\lim_{\xi\rightarrow +\infty} \psi(\xi,y) & = & \psi^0(y),  \end{array} \right. \label{full-pr} \end{equation}
 with $v>0$ in $\Omega_Y$. Furthermore, $v$ satisfies the following properties.
 \begin{itemize}
 \item Behaviour close to $\xi=0$: for some $\xi_0,~m>0$, depending on $Y$, we have $\derp{\xi}v>m$ for $\xi\leq \xi_0$.
 \item Behaviour close to $\xi=\infty$: for every $y$, $v(\xi,y)\sim a(y) (\psi^0(y)-\psi(\xi, y))$ as $\xi\to \infty$.
\end{itemize}
 
 \label{exist}
\end{theo}

\begin{rmk}
\begin{itemize}
\item The condition on the lower bound of $\lambda_0 (\psi^0)' \nu^{-1/3}$ is more stringent than the one suggested by the analysis of the linearised equation \eqref{linearized}.
 We will see that this condition stems from the necessity of having a good control on the derivatives of $v$ (see Proposition \ref{alterd}).

\item Conditions \eqref{comppr} and \eqref{comp-infty} are compatibility conditions on the initial data close to $\xi=0$ and $\xi=\infty$. Our result shows that these properties are propagated by the equation.
 Notice also that the decay condition at infinity is the one that is expected from the analysis of the linearised equation \eqref{linearized}, and it implies rapid (exponential) decrease in the variable $\xi$.
 This last fact is explained in the proof of \cororef{equivcoro}.

 \item Notice that our theorem requires both $\psi^0$ and $(\psi^0)'$ to be strictly positive. This is consistent with the analysis of \cite{Pjbook} in which $\psi^0$ is linear with a positive slope.
  Once again, we believe that such an assumption is necessary to obtain the positivity of $v$. Indeed, the important quantity is really $\lambda_0 (\psi^0)'$,
  so that having a negative $(\psi^0)'$ and a positive $\lambda_0$ is more or less equivalent to having a positive $(\psi^0)'$ and a negative $\lambda_0$,
  and the work of Ierley and Ruehr \cite{IerleyRuehr} gives numerical and analytic evidence that when $\lambda_0$ is negative and $\psi^0$ is linear, with $\nu\equiv 1$, the function $v$ takes negative values.
 
 \end{itemize}
 We will make further remarks on the physical meaning of our conditions in the next paragraph.

\end{rmk}

Looking for a positive solution to the problem implies, by definition of the stream function, that $\psi$ is monotonous, increasing in the variable $\xi$.
 Thus, the von Mises transform, used originally on the ordinary steady Prandtl equation (\ref{pr-usual}) by O. Ole\u{i}nik \cite{Oo63}, fits the setting we have introduced.
 The change turns the unknown function $\psi$ into an independent variable, and the Prandtl equation becomes a PDE for $w(\psi,y)=v^2(\psi(\xi,y),y)$ on the finite domain
 $$ D = \{(\psi,y)\in \Rr^2 ~|~ 0<y<Y \mathand{and} 0<\psi<\psi^0(y) \} . $$
 The parabolic equation we obtain is degenerate, in the sense that the square root of the unknown $w$ appears as a factor in the diffusive term,
 and this function is assumed to vanish on the left- and right-hand boundaries of $D$.
 An approximate problem is considered, and we construct explicit super- and sub-solutions in order to get uniform Lipschitz bounds,
 allowing us to get a solution of the degenerate problem by taking the limit of these approximate solutions.

The proof given in sections 2-4  follows the plan of O. Ole\u{i}nik and V. Samokhin \cite{OSbook}, and what happens near the left-hand boundary of $D$, the set
 $$ \{(\psi,y)~|~0\leq y\leq Y \mathand{and} \psi=0\} , $$
 is identical. \textit{The new difficulty for our case is the presence of a right-hand boundary,
 $$ \{(\psi,y)~|~0\leq y\leq Y~~\text{and}~~ \psi = \psi^0(y)\} , $$
 and having to control what happens there.} Indeed, we need $v$ to be exponentially decreasing in $\xi$, and $\psi$ itself should converge exponentially towards $\psi^0$; we therefore expect $w(\psi,y)$ to be
 $\Oc((\psi^0-\psi)^2)$. The link between these behaviours is shown below.

\subsection{Physical derivation of equation \eqref{pr}}
\label{ssec:physique}

The starting point is the 3D homogeneous, incompressible Navier-Stokes-Coriolis equations with anisotropic viscosity, which describe the motion of oceanic currents at mid-latitudes on large horizontal scales
 ($100-1000 \; \mathrm{km}$). In an appropriate scaling (fast rotation, thin layer domain, small vertical viscosity), it can be proved that the fluid behaves in the limit like a 2D fluid;
 we refer for instance to the study by Desjardins and Grenier \cite{DG}. The equation describing the motion then becomes
\be\label{NS2d}\begin{aligned}
(\p_t +  \bar u_1 \p_x +  \bar u_2\p_y) \left(\lambda\zeta + \beta y + \eta_B\right) =\beta \curl \tau + \frac{1}{Re}\Delta \zeta - \frac{r_0}{2}\zeta, \quad \text{in } (0,\infty)\times \Omega\\
\p_x \bar u_1 +\p_y \bar u_2=0 , \quad \text{in } (0,\infty)\times \Omega
\end{aligned}
\ee
where
\begin{itemize}
\item $\zeta= \curl (\bar u_1,\bar u_2)$ is the vorticity;
\item $\Omega$ is a two-dimensional domain representing the oceanic basin;
\item $\beta>0$ is a parameter characterising the beta-plane approximation (linearisation of the Coriolis factor around a given latitude);
\item $- \frac{r_0}{2}\zeta$ is the Ekman pumping term due to friction on the bottom, and $\beta \curl \tau$ is the Ekman pumping term due to the wind stress at the surface ($\tau $ is a given function);
\item $\eta_B$ describes the variations of topography of the bottom;
\item $Re$ is the Reynolds number.
\end{itemize}
Equation \eqref{NS2d} is endowed with no-slip boundary conditions:
$$
\bar u_{|\p \Omega}=0.
$$
Let us assume that the domain $\Omega$ is of the form
$$
\Omega=\{(x,y)\in \Rr^2,\ \chi(y)<x<\sigma (y), \ 0<y<Y\},
$$
where the functions $\chi$ and $\sigma$ are given and smooth. We now look at the asymptotic behavior of equation \eqref{NS2d} in the limit 
$$
\lambda\ll 1, \quad r_0\ll 1,\quad Re\gg 1. 
$$
In order to simplify the analysis, we assume that the bottom is flat, so that $\eta_B=0$. We look for stationary approximate solutions of \eqref{NS2d}.
 In the limit $\lambda, r_0, (Re)^{-1}\to 0$, we expect the solution in the interior to satisfy the Sverdrup relation,
$$
\beta  u_2^{int}= \beta \curl \tau ,~~ \text{i.e.}~~ \p_x  \psi^{int}= \curl \tau,
$$
where the stream function $\psi^{int}$ is defined by $u^{int}= \nabla^\bot \psi^{int}$.
 However, such a function cannot satisfy all the boundary conditions: it is well-known that a dissymmetry occurs between western and eastern boundaries
 (see \cite{Pjbook,DG} and the further analysis in \cite{DSR}) and that $\psi^{int}$ should vanish on the eastern boundary, so
$$
\psi^{int}(x,y)=-\int_{x}^{\sigma(y)} \curl \tau.
$$
However, in general $\psi^{int}$ does not vanish on the western boundary. Therefore a boundary layer is created close to the western coastline,
 in order to satisfy the no-slip boundary conditions $\psi_{|\p \Omega}=0, ~ \p_n \psi_{|\p\Omega}=0$. The nature of this boundary layer depends on the sizes of the parameters $\lambda, Re, r_0$.
 Notice that in the stationary case, equation \eqref{NS2d} can be re-written as
\be\label{eq:psi}
\lambda \nabla^\bot \psi \cdot \nabla \Delta \psi + \beta \p_x \psi + \frac{r_0}{2} \Delta \psi - \frac{1}{Re} \Delta^2 \psi= \beta \curl \tau.
\ee
An extensive discussion around the sizes of the parameters and the corresponding governing equations for the western boundary layer can be found in \cite{Pjbook}.
 So far, linear Munk layers, which formally correspond to the case $r_0\ll (Re)^{-1/3}\ll 1$, $\lambda \ll (Re)^{-2/3}\ll 1$ have been thoroughly studied (see for instance \cite{BC,DSR};
 we also refer to \cite{BGV} for the study of a non-linear version of this problem when $\lambda \sim (Re)^{-1}$ in the presence of a rough boundary; in this setting, the boundary layer equation remains elliptic,
 which will not be the case in our study). Stommel layers, which correspond to the case $(Re)^{-1}\ll r_0^3\ll 1$, $\lambda\ll r_0^2\ll 1$, have  been studied in \cite{DG};
 we also refer to \cite{BCT}, in which Barcilon, Constantin and Titi prove the existence of weak solutions and of stationary solutions of \eqref{NS2d} (in the vorticity formulation)
 in the regime $Re=+\infty$, $r_0\ll 1$.

\textit{Here, we want to study the influence of the nonlinear term in the case of a smooth western boundary.} The correct scaling to retain the nonlinear term is therefore $\lambda \varpropto (Re)^{-2/3}$.
 We therefore take $r_0=0$, $\beta=1$ and $\lambda=\lambda_0 (Re)^{-2/3}$ for some parameter $\lambda_0>0$ which we will keep throughout our study.
 A formal analysis shows that the boundary layer size is expected to be $(Re)^{-1/3}$. Hence we plug an ansatz of the form $\psi=\psi^{BL}((x-\chi(y))(Re)^{1/3}, y)$ into \eqref{eq:psi},
 so that, when $Re\rightarrow +\infty$,
\begin{eqnarray*}
\Delta \psi & \sim & (Re)^{2/3}(1+ \chi'(y)^2) \p_{\xi}^2 \psbl ((x-\chi(y))(Re)^{1/3}, y),\\ \Delta^2 \psi & \sim & (Re)^{4/3}(1+ \chi'(y)^2)^2 \p_{\xi}^4 \psbl ((x-\chi(y))(Re)^{1/3}, y).
\end{eqnarray*}
We find that the equation satisfied by $\psi^{BL}=\psi^{BL}(\xi, y)$ is
$$
\lambda_0 (1+\chi'(y)^2)\left[\p_{\xi} \psbl \psi \p_y \p_{\xi}^2 \psbl -  \p_y \psbl \p_{\xi}^3 \psbl \right] + \p_{\xi}\psbl - (1+\chi'(y)^2)^2\p_\xi^4 \psbl=0,
$$
with no-slip boundary conditions $\psbl_{|\xi=0}=\p_{\xi} \psbl_{|\xi=0}=0$. In the limit $\xi\to \infty$, i.e. at the end of the boundary layer,
 $\psbl\to \psi^0(y):= -\int_{\chi(y)}^{\sigma(y)} \curl \tau(x,y)\:dx$. Noticing that
$$
\p_{\xi} \psbl \psi \p_y \p_{\xi}^2 \psbl -  \p_y \psbl \p_{\xi}^3 \psbl=\p_{\xi}\left[\p_\xi \psbl \p_y\p_\xi \psbl - \p_y \psbl \p_{\xi}^2 \psbl\right],
$$
we infer that the equation becomes
$$
\lambda_0 (1+\chi'(y)^2)\left[\p_\xi \psbl \p_y\p_\xi \psbl - \p_y \psbl \p_{\xi}^2 \psbl \right] + \psbl - (1+\chi'(y)^2)^2\p_\xi^3 \psbl=\psi^0(y).
$$
Now, setting $(u,v)=(-\p_y \psbl, \p_{\xi} \psbl)$ (the scaled velocity in the boundary layer), we obtain
\be\label{pr-chi}
\lambda_0 (1+\chi'(y)^2)\left[ u \p_{\xi } v + v \p_{y} v  \right] + \psbl - (1+\chi'(y)^2)^2\p_\xi^2 v =\psi^0(y),
\ee
with the boundary conditions
\begin{eqnarray*}
v_{|\xi=0} = u_{|\xi=0} = \psbl_{|\xi=0} & = & 0,\\
\lim_{\xi\to \infty} \psbl(\xi,y) & = & \psi^0(y),\\
\lim_{\xi\to \infty} v(\xi,y) & = & 0.\end{eqnarray*}
Dropping the $BL$ superscript in $\psbl$, we obtain equation \eqref{pr} with $\nu(y):= (1+\chi'(y)^2)$. \textbf{The function $\nu $ therefore describes the geometry of the coastline.}

\begin{rmk}[Physical meaning of the conditions in Theorem \ref{exist}] $ $
\begin{itemize}
\item The condition $\lambda_0 (\psi^0)'(y)/\nu(y)^{1/3} > 2+\eta$  fails as $\chi'$ becomes infinite, for example when $y$ reaches the South and North extremities of the ocean basin.
 But it is well known that even the derivation of classical linear Munk layers fails in this case, and that the nature and the size of the boundary layer changes in the vicinity of these points (see \cite{DSR}).
\textbf{ Our result also shows that a boundary layer separation could be caused by a sudden change in the coastline},
 since the quantity $\lambda_0 (\psi^0)'\nu^{-1/3}$ would dip below the critical threshold which ensures that there is no recirculation.

\item From a physical point of view, the sign of $\psi^0$ is related to the direction of the rotation of winds and currents in the underlying gyre, since $\psi^0$ is the integral of $\curl \tau$.
 The positivity of $(\psi^0)'$ means that the rotation of winds is stronger at higher latitudes. The size of $\lambda_0$ is related to the intensity of the non-linear effects compared to the rotation.
 Therefore our result could be summarized as: ``in a rotating gyre, if the intensity of inertial forces is strong enough and the coastline has slow variations, there is no recirculation,
 and no separation phenomenon in the western boundary layer''.

\item We have formulated our result on a bounded interval $(0,Y)$, keeping in mind the physical picture of an ocean basin between two fixed latitudes, but from a mathematical point of view,
 we could have assumed that $Y=+\infty$, thereby constructing global solutions. Notice also that because of the condition on $\nu$, the $y$-interval is actually smaller than that of the actual ocean basin:
 we must exclude the zone close to the North and South extremities.
\end{itemize}
\end{rmk}

%
%
%
%

\section{Reformulation of the equation and statements in von Mises variables}

\label{ssec:transfo}

Fix $Y>0$ for the rest of the paper. Starting with equation (\ref{full-pr}), we follow \cite{OSbook} and make the von Mises change of variables,
\be\label{cgt-var} (\xi,y) \mapsto (\psi(\xi,y),y) , \ee
and choose a new unknown function, $w=w(\psi,y)$ such that $w(\psi, y) = v^2(\psi(\xi,y),y)$. The chain rule allows us to express the derivatives: we have
\begin{equation}
\begin{aligned}
\frac{\p v}{\p \xi}= \frac{1}{2}\frac{\p w}{\p\psi},\quad \frac{\p^2 v}{\p \xi^2}=\frac{\sqrt{w}}{2}\frac{\p^2 w}{\p \psi^2}, \\
\frac{\p v}{\p y}=\frac{1}{2\sqrt{w}}\frac{\p w}{\p y} + \frac{1}{2\sqrt{w}}\frac{\p w}{\p \psi}\frac{\p \psi}{\p y}.
\end{aligned} \label{chainrule}
\end{equation}
 Therefore the equation on $w$ is the degenerate parabolic PDE
\begin{equation} \left\{ \begin{array}{rcl} \lambda_0 \nu(y)\derp{y}w - \nu(y)^2\sqrt{w}\derp{\psi}^2 w & = & 2(\psi^0(y)-\psi) \\
w|_{y=0} & = & w_0(\psi) \\
w|_{\psi=0} = w|_{\psi=\psi^0(y)} & = & 0 \end{array} \right. \label{vm} \end{equation}
on the domain $D$, which, we recall, is
$$ D = \{(\psi,y)\in \Rr^2 ~|~ 0<y<Y \mathand{and} 0<\psi<\psi^0(y) \} . $$
The corner compatibility condition for the initial datum $w_0$ is deduced from (\ref{comppr}),
\begin{equation} \nu(0)^2\sqrt{w_0(\psi)}w_0''(\psi) + 2(\psi^0(0)-\psi) = \Oc(\psi)\quad \text{for } \psi\ll 1 . \label{compvml} \end{equation}
Concerning the behaviour of $w$ as  $\psi\rightarrow \psi^0(y)$, we have
\begin{equation} w_0(\psi) \sim a(0)^2 (\psi^0(0)-\psi)^2 \quad \text{as } \psi \to \psi^0(0) . \label{compvmr} \end{equation}

The goal of the next  sections is to prove the two following results: the first one deals with existence and uniqueness of solutions of \eqref{vm}, together with some qualitative properties.
\begin{propo} Let $\psi^0\in \mathcal C^2([0,Y])$ be a positive increasing function. Assume that $\lambda_0 (\psi^0)'(y) \nu(y)^{-1/3}> 2+\eta>\sqrt[3]{\frac{27}{4}}$ for every $y\geq 0$ and for some $\eta>0$.
 Let $w_0:[0,\psi^0(0)]\to \Rr_+$ be a $\mathcal{C}^2$ function satisfying \eqref{compvml} and \eqref{compvmr}.

Then there exists a unique classical solution $w$ of equation \eqref{vm} such that $w$ is strictly positive in the interior of $D$. Furthermore, $w$ satisfies the following properties.
\begin{itemize}
\item Behaviour near $\psi=0$: $\p_\psi w(y,0)>0$ for all $y\in [0,Y]$ and there exists $\beta \in (0,1/2)$, $C>0$ such that
$$
|\p_y w(\psi, y)|\leq C \psi^{1-\beta} 
$$
in a neighbourhood of $\psi=0$.
\item Boundedness of the derivatives: $\p_\psi w$ and $\p_y w$ are bounded in $\bar D$.
\item Behaviour near $\psi=\psi^0(y)$: as $\psi\to \psi^0(y)$,
$$
w(\psi,y)\sim a(y)^2 \left( \psi^0(y)-\psi\right)^2,
$$
and there exist $\Cc^1$ functions $E^\pm(y)>0$ and a constant $C_0>0$  such that
\begin{eqnarray*} -E^-(y)(\psi^0(y)-\psi) \leq & \p_\psi w(\psi,y) & \leq -E^+(y)(\psi^0(y)-\psi)  \\
\text{and}~ & |\p_y w(\psi,y)| & \leq C_0(\psi^0(y)-\psi)
\end{eqnarray*}
in a neighbourhood of the boundary $\{\psi = \psi^0(y)\}$.
\end{itemize}

\label{prop:exvm}

\end{propo}

The second result ensures the existence of a solution to the rotating Prandtl equation \eqref{full-pr}.
\begin{coro}
Assume that there exists a solution of \eqref{vm} with the properties listed in Proposition \ref{prop:exvm}. Define $v=v(\xi,y)$ through the inverse change of variables to \eqref{cgt-var},
 and $u$ as the $y$-derivative of the stream function $\psi$. Then $(u,v,\psi)$ solve \eqref{full-pr} with the properties mentioned in \thref{exist}.
\label{equivcoro}
\end{coro}

We will show the uniqueness part of \thref{exist} in section \ref{ssec:conclu}, after proving uniqueness for equation (\ref{vm}).

\begin{proof}[Proof of \cororef{equivcoro}]
The reverse change of variables is
\begin{equation} (\xi,y) = \left(\int_0^{\psi(\xi,y)} \frac{1}{\sqrt{w(s,y)}}~ds , y\right) . \label{reverse-cgt-var} \end{equation}
 We first notice that, given the behaviour of $w$ near $s=0$ and $s=+\infty$, the above integral is convergent for $\xi < +\infty$,
 and that the limit $\psi\rightarrow \psi^0(y)$ indeed corresponds to $\xi\rightarrow +\infty$.
\gap

Differentiating (\ref{reverse-cgt-var}) with respect to $\psi$, we notice that $\derp{\psi}\xi = 1/\sqrt{w(\psi,y)}$, which implies that $\derp{\xi}\psi = v$.
 The boundary and compatibility conditions are obviously satisfied, but let us explain why the decay condition (\ref{comp-infty}) implies
 that $v(\xi,y)$ is exponentially decreasing in $\xi$. For a fixed $y$ and a small $\delta>0$, there exist $0<c<a(y)<C$ such that for $\psi^0(y)-\delta<s<\psi^0(y)$,
 $$ c(\psi^0(y)-s) \leq \sqrt{w(s,y)} \leq C(\psi^0(y)-s) . $$
 Integrating from $\psi^0(y)-\delta$ to $\psi(\xi,y)$, we get that $\xi$ is logarithmic in $\psi$, and we obtain
 $$ K_\delta e^{-C\xi} \leq \psi^0(y)-\psi(\xi,y) \leq K_\delta e^{-c\xi}. $$
 Inserting this in the equivalent for $v$ as $\psi\rightarrow \psi^0(y)$, we see the rapid decrease property.
 \gap

Differentiating the expression of $\xi$ in (\ref{reverse-cgt-var}) with respect to $y$, we obtain a formula for $-\derp{y}\psi=u$:
\begin{equation} u(\xi,y) = -\derp{y}\psi = -\frac{1}{2} \sqrt{w} \int_{0}^{\psi(\xi,y)} \frac{\derp{y}w(s,y)}{w^{3/2}(s,y)}~ds. \label{defu} \end{equation}
With this, we can determine the limits as $\xi\rightarrow 0$ and $+\infty$ of $u$. Indeed, on one hand, we have the fact that for $0<s\ll 1$,
$$w^{-3/2}(s,\psi)\derp{y}w(s,\psi) = \Oc( s^{-3/2}\psi^{1-\beta})= \Oc(s^{-1/2-\beta}),$$ and $1/2+\beta < 1$.

On the other hand, by Proposition \ref{prop:exvm}, as $\xi\rightarrow +\infty$, we have $\derp{y}w(s,y) = \Oc( (\psi^0(y)-s))$ and $w = \Oc( (\psi^0(y)-s)^2)$, so
$$ \frac{\derp{y}w(s,y)}{w^{3/2}(s,y)} = \mathop{\Oc} ((\psi^0(y)-s)^{-1/2})\quad \text{as }s\rightarrow \psi^0(y). $$
In both cases, the integral in (\ref{defu}) is shown to be convergent for any $\psi$, and is therefore bounded. The factor $\sqrt{w}$ then shows that $u$ satisfies the zero boundary condition at $\xi=0$
 and the limit condition at $\xi\rightarrow+\infty$. With the identities in (\ref{chainrule}), we then see that $(u,v,\psi)$ solve the Prandtl equation in a classical way.
\end{proof}

The next sections are devoted to the proof of Proposition \ref{prop:exvm}. We first recall some notation and useful results. 
In the rest of the article, we will call ``left-hand boundary'' the boundary $\{\psi=0\}$, and ``right-hand boundary'' the boundary $\{\psi=\psi^0(y)\}$.
 We will strongly rely on the fact that the functional map in equation (\ref{vm}),
$$ \Lc: f \mapsto \lambda_0 \nu(y)\derp{y}f - \nu(y)^2\sqrt{f} \derp{\psi}^2 f, $$
has maximum and comparison principles that we state below, providing $f$ does not vanish.
 We send the reader to Ole\u{i}nik \& Samokhin for the proof (it uses a standard contradiction method on the assumption of an interior extremum).

\begin{lemma}[\cite{OSbook}, Lemmata 2.1.2 and 2.1.3]
Let $\Dc$ be an open subset of $\Rr^2$ of the form
$$ \Dc = \{(\psi,y)~|~ 0<y<Y ~~\text{and}~~ L(y)<\psi<R(y)\}, $$
and $\Gamma$ be the parabolic boundary of $\Dc$,
$$ \Gamma = \{(\psi,y)\in\overline{\Dc}~|~\psi=L(y) ~~\text{or}~~ y=0 ~~ \text{or}~~ \psi = R(y)\}. $$
Consider a differential operator of the form
$$
\Ac(w)=\p_y w + b(y,w)\p_\psi w
 + c(\psi,y)w - d(y,w) \p_{\psi}^2 w,$$
in which $d(y,w), b(y,w), c(\psi,y)$ are bounded for $(\psi,y)\in \overline{\Dc}$, and $w$ bounded, $d$ and $b$ have bounded partial derivatives with respect to $w$,
 and $d$ takes strictly positive values. Let $f$ and $g$ be continuous functions in $\overline{\Dc}$ such that their derivatives occurring in $\Ac$ are continuous in $\overline{\Dc}\setminus \Gamma.$

Assume that
$$ \Ac(f)\geq \Ac(g) ~\text{in}~ \Dc ~~\text{and}~~ f\geq g ~\text{on}~ \Gamma. $$

Then $f\geq g$ in $\overline{\Dc}$.
\label{comparl}
\end{lemma}
\begin{rmk}
Actually, in the version of the Lemma stated in \cite{OSbook}, the function $b$ does not depend on $w$ and the function $d$ does not depend on $y$. 
 However, it can be easily checked that the proof of \cite{OSbook} can be immediately extended to the setting of Lemma \ref{comparl}.
\end{rmk}

We will construct a solution to problem (\ref{vm}) by considering a sequence of approximate solutions with strictly positive boundary data.

In the unbounded domain $\rplus\times(0,Y)$, like in the classical Prandtl equation, one would simply shift the initial data slightly to the left.
 In our case, we must also account for the right-hand border. What we do is we truncate the domain.
\gap
 
From now on, we assume that the condition $\lambda_0 (\psi^0)'\nu^{-1/3}>\sqrt[3]{27/4}$ is satisfied, so that the function $a(y)$, the first positive root of the polynomial $P_y$, is well-defined.
 Let $\eps>0$ be small. We consider the approximate domain
 $$ D^\eps = \{(\psi,y)\in\Rr^2 ~|~ 0 < y < Y \mathand{and} 0 < \psi < \psi^0(y)-2\eps \}, $$
 its  parabolic boundary 
 $$ \Gamma^\eps = \{(\psi,y)\in \overline{D^\eps}~|~ \psi = 0 \mathand{or} y=0 \mathand{or} \psi = \psi^0(y)-2\eps\} $$
  and the approximate boundary-value problem
\begin{equation} \left\{ \begin{array}{rcl} \lambda_0 \nu(y)\derp{y} w^\eps -\nu(y)^2\sqrt{w^\eps}\derp{\psi}^2 w^\eps & = & 2(\psi^0(y)-\psi-\eps) \\
w^\eps|_{y=0} & = & w_0(\eps+\psi) \\
w^\eps|_{\psi=0} & = & w_0(\eps)\exp\left(\frac{\mu(\eps)y}{w_0(\eps)}\right) \\
w^\eps|_{\psi=\psi^0(y)-2\eps} & = & w_0\left(\psi^0(0) - \eps\right) \frac{a(y)^2}{a(0)^2} , \end{array} \right. \label{vma} \end{equation}
where $\mu(s) = \nu(0)^2\sqrt{w_0(s)}w_0''(s)+2(\psi^0(0)-s)$, which is $\Oc (s)$ when $s\rightarrow 0$ by (\ref{compvml}).
 The right-hand boundary condition replicates our expectation, which stems from the study of ODE (\ref{ode}), that $w^\eps(\psi,y)$ will resemble $a(y)^2 (\psi^0(y)-\psi-\eps)^2$.
  Let us denote the approximate boundary data
\begin{eqnarray*}
 w^\eps_l(y) & := & w_0(\eps)\exp\left(\frac{\mu(\eps)y}{w_0(\eps)}\right),\\
\text{and}~~ w^\eps_r(y) & := & w_0\left(\psi^0(0) - \eps\right) \frac{a(y)^2}{a(0)^2} .\end{eqnarray*}

The goal of the next section will be to prove the following uniform bounds on solutions of the approximate equations.

\begin{propo}
There exist constants $\delta, ~ C_0>0$ such that, for every $0\leq y\leq Y$:
\begin{itemize}
\item near the left-hand boundary, i.e. in the zone $0<\psi<\delta$, there exist $A^\pm >0$ and $B^+>B^->0$, such that for $\eps>0$ small enough,
\begin{eqnarray*}
 w_l^\eps(y) + A^- \psi^{4/3} + B^-\psi \leq & w^\eps(\psi,y) &  w_l^\eps(y) +  + B^+\psi-A^+ \psi^{4/3} , \\
 C_0^{-1} \leq & \derp{\psi}w^\eps (\psi,y) & \leq C_0 , \\
  & \left| \p_y w^\eps(\psi,y)\right| & \leq C_0 ; \end{eqnarray*}

\item near the right-hand boundary, i.e. in the zone $ \psi^0(y) -\delta <\psi <\psi^0(y)-2\eps$, there exist $\Cc^1$ functions $C^\pm(y)$ with $0<C^-(y)<a(y)<C^+(y)$, $E^\pm(y)>0$, such that for $\eps>0$ small enough,
\begin{multline}\label{est:rhb}
 C^-(y)^2 (\psi^0(y) -\psi-\eps)^2+ w_r^\eps(y) - C^-(y)^2 \eps^2  \leq w^\eps(\psi,y)\\\leq  C^+(y)^2 (\psi^0(y) -\psi-\eps)^2 + w_r^\eps(y) - C^+(y)^2 \eps^2
\end{multline}
and
\begin{eqnarray*}
 - E^-(y)\left( \psi^0(y) -\psi-\eps\right) \leq & \p_\psi w^\eps(\psi,y) & \leq - E^+(y)\left( \psi^0(y) -\psi-\eps\right) ,\\
  -C_0\left( \psi^0(y) -\psi-\eps\right) \leq & \p_y w^\eps(\psi,y) & \leq C_0\left( \psi^0(y) -\psi-\eps\right); \end{eqnarray*}

\item and in the middle zone $\delta \leq \psi \leq \psi^0(y) -\delta $,
$$
\ba 
C_0^{-1}\leq w^\eps(\psi,y)\leq C_0,\\
|\p_y w^\eps(\psi,y)|,\ |\p_\psi w^\eps(\psi,y)|\leq C_0.
\ea
$$

\end{itemize}

All the parameters defined in the proposition depend on $Y$.
\end{propo}

\begin{rmk}
In the above proposition, the functions $C^\pm(y)$ can be chosen as close to $a(y)$ as desired. Choosing $C^\pm$ close to $a$ will simply reduce the size of the zone in which \eqref{est:rhb} is valid, $\delta$.
 In a similar way, the functions $E^\pm$ can be chosen as close to $2 a^2$ as desired.
\end{rmk}

In the case of the $L^\infty$ bounds, we will create functions $\Phi^-_\eps$ (resp. $\Phi^+_\eps$) showing the above behaviour, such that $\Phi^-_\eps\leq w^\eps$ (resp. $\Phi^+_\eps\geq w^\eps$) on the parabolic boundary of $D^\eps$,
 and $\Lc \Phi^-_\eps \leq \Lc w^\eps$ (resp. $\Lc \Phi^+_\eps \geq \Lc w^\eps$) in the interior of $D^\eps$. \lemref{comparl} (b) then gets us the lower (resp. upper) bounds on $w^\eps$.
 
\begin{nota}
From here, to condense the writing, we will use the symbols $\pm$, $\gtrless$ and $\lessgtr$. In equalities or estimates involving these combined symbols, we mean that the property with just the upper symbols is true,
 and that, respectively, the property with just the lower symbols is also true.

For example, the above discussion on $\Phi^-_\eps$ and $\Phi^+_\eps$ can be abridged as follows: we construct $\Phi^\pm_\eps$ such that $\Phi^\pm_\eps\gtrless w^\eps$ on the parabolic boundary of $D^\eps$,
 and $\Lc \Phi^\pm_\eps \gtrless \Lc w^\eps$ on $D^\eps$. \lemref{comparl} (b) implies that $\Phi^\pm_\eps \gtrless w^\eps$ on $D^\eps$.
\end{nota}
 
 \section{Derivation of a priori bounds on the approximate problem}
 
\subsection{$L^\infty$ bounds}

The derivation of $L^\infty$ bounds is in two steps: 
\begin{itemize}
\item First, we construct sub- and super-solutions which will not have the precise desired behaviour near the boundaries, but that are \textbf{global} on the interval $[0,Y]$.
 We will refer to these functions as ``blanket sub-/super-solutions''.

\item Then we construct refined sub- and super-solutions close to the right-hand boundary. This construction is \textbf{local} in $y$, but the blankets will cover the shortcomings:
 thanks to the \textit{a priori} global bounds derived in the first step, we are able to iterate our construction and to get a control of $w^\eps$ close to the right-hand boundary over the whole interval $[0,Y]$.
\end{itemize}

We start with a result which we will use throughout this paragraph, and which highlights the behaviour of $w^\eps$ close to the right-hand boundary.

\begin{propo}
Assume that $\inf_y \lambda_0 (\psi^0)'(y)\nu(y)^{-1/3}>\sqrt[3]{27/4}$. Consider $\Cc^1$ functions $C^\pm(y)$ such that for all $y\in [0,Y]$,
$$
0<C^-(y)<a(y)<C^+(y)< \frac{2 \lambda_0 (\psi^0)'(y)}{3\nu(y)},
$$
and define
$$
W^\pm_\eps:=C^\pm(y)^2(\psi^0(y)-\psi-\eps)^2 + (w_r^\eps(y)- C^\pm(y)^2 \eps^2).
$$

Then there exist $\gamma^\pm \gtrless 0$ and $\delta^\pm>0$ such that for all $\eps>0$ sufficiently small, on the domain
$$ \{(\psi, y) ~|~ 0\leq y \leq Y,~ \psi^0(y)-\delta^\pm \leq \psi \leq \psi^0(y)-2\eps\}, $$
there holds
$$
\Lc W^\pm_\eps \gtrless (2+\gamma^\pm) (\psi^0(y)-\psi-\eps).
$$

Furthermore, if $C^-(y)$ is chosen in a neighbourhood of zero, then $\gamma^-<-\gamma_0$ and $\delta>\delta_0$, for some universal positive constants $\gamma_0$, $\delta_0$.
\label{prop:rhb}
\end{propo}

\begin{proof}
We have
 \begin{equation} \begin{array}{rcl} \Lc W^\pm_\eps &=& 2\Big(\lambda_0 \nu(y) (\psi^0)'(y) (C^\pm)^2 - \nu(y)^2(C^\pm)^3\Big)(\psi^0(y)-\psi-\eps) \\
 & &+ 2\lambda_0\nu(C^\pm)'(y)C^\pm(y)(\psi^0(y)-\psi-\eps)^2+\lambda_0\nu[\p_y w_r^\eps(y)-2(C^\pm)'(y)C^\pm(y) \eps^2], \end{array} \label{lphi3}
  \end{equation}
 and the main property we want is, taking the top line of (\ref{lphi3}),
 $$ \lambda_0 \nu(y)(\psi^0)'(y) (C^\pm)^2 - \nu(y)^2(C^\pm)^3 - 1 = -P_y(C^\pm) \gtrless 0. $$
 As $\lambda_0 (\psi^0)'(y)/\nu^{-1/3}$ is assumed to be larger than $\sqrt[3]{27/4}$, $P_y$ has two positive roots, the first of which is $a(y)$, and $-P_y$ is increasing on $[0, 2 \lambda_0 (\psi^0)'(y)/(3\nu(y))]$.
 As a result, providing $C^+$ is taken strictly between the functions $a(y)$ and $2 \lambda_0 (\psi^0)'(y)/(3\nu(y))$, and that we choose $0<C^-(y)<a(y)$, we get, for some $\gamma^+>0$ and $\gamma^-<0$,
 $$
2\Big(\lambda_0 \nu(y) (\psi^0)'(y) (C^\pm)^2 - \nu(y)^2(C^\pm)^3\Big) \gtrless 2+2\gamma^\pm.
 $$
 Notice that the functions $C^\pm$ can be chosen to be as regular as desired, and that if $C^-$ is in a neighbourhood of zero, then $\gamma^-<-1/2$.
 
We now consider the lower order terms, i.e. the ones involving the derivatives of $C^\pm$ and of $w_r^\eps$.
 First, notice that $2\lambda_0 \nu(y)(C^\pm)'(y)C^\pm(y)(\psi^0(y)-\psi-\eps)^2$ vanishes at $\psi^0(y)-\eps$ at a higher order than the first line of (\ref{lphi3}).
 We therefore choose $\delta^\pm$ so that
$$
\delta^\pm\leq \inf_{y\in Y} \frac{|\gamma^\pm|}{4 \lambda_0 \nu(y) |(C^\pm)'(y)| C^\pm(y)}.
$$
Once again, if $C^-$ is in a sufficiently small neighbourhood of zero in $\mathcal C^1([0,Y])$, one can take $\delta^-=1$.

Then, we recall that $\p_y w_r^\eps(y)= 2 a(y) a'(y) w_0(\psi^0(0)-\eps)/a(0)^2=\Oc(\eps^2)$, so
$$
\left| \lambda_0\nu(y)\p_y w_r^\eps(y) - 2\lambda_0 \nu(y) (C^\pm)'(y)C^\pm(y) \eps^2\right|\leq C \eps^2\leq \frac{|\gamma^\pm|}{2}(\psi^0(y)-\psi - \eps)
$$
for all $\psi\in [\psi^0(y)-\delta^\pm, \psi^0(y)-2\eps]$, provided $\eps$ is sufficiently small.

In total, if $\psi \in [\psi^0(y) -\delta^\pm, \psi^0(y)-2 \eps]$,
$$\Lc W^\pm_\eps \gtrless (2+\gamma^\pm)(\psi^0(y)-\psi-\eps).$$
\end{proof}

\subsubsection{``Blanket'' sub- and super-solutions}

Let $\delta_0>0$ to be chosen later, and we divide $D^\eps$ into three zones:
 $$ \begin{array}{rcl} I & = & \{ 0< \psi < \delta_0 \} \\
II & = & \{ \delta_0 \leq \psi \leq \psi^0(y)-\delta_0 \} \\
III & = & \{ \psi^0(y)-\delta_0 < \psi < \psi^0(y)-2\eps \}, \end{array} $$
and $(h_I, h_{II}, h_{III})$ a partition of unity such that $h_I$ is nonincreasing, $h_{III}$ is nondecreasing, $h_I\equiv 1$ in zone $I$, $h_{III} \equiv 1$ in zone $III$,
 and $h_{II}\equiv 1$ in the set
 $$ \{(\psi,y)\in II ~|~ 2\delta_0 \leq \psi \leq \psi^0(y)-2\delta_0 \} . $$
 We shall see that we will be able to use Ole\u{i}nik and Samokhin's constructs on the first two zones, while in the third zone, the blanket sub-solution will be given by $W^-_\eps$ with $C^-$ small.
 This smallness is required to control the transition between zones $II$ and $III$.

\begin{propo}\label{prop:blankets}
Assume that a solution to problem (\ref{vma}) exists. 

\begin{itemize}
\item There exists a constant $\overline{M}$, depending only on $Y$, $\|w_0\|_\infty$, $\psi^0$ and $\nu$, such that
$$
w^\eps(\psi, y)\leq \overline{M},\quad \forall (\psi,y)\in D^\eps.
$$

\item Define a function $\wme$ by
\begin{equation} \begin{array}{rcl} \wme (\psi,y) & = & h_I(\psi)\left[ w^\eps_l(y)+\left(\underline{A} \psi^{4/3} + \underline{B} \psi\right)e^{-\alpha_0 y} \right] \\
& & + h_{II}(\psi,y) \underline{M} e^{-\alpha_0 y} \\
& & + h_{III}(\psi,y) W^-_\eps, \end{array} \label{asolw} \end{equation}
with the positive parameters $\underline{A}$, $\underline{B}$, $\underline{M}$ and $\alpha_0$. 

Then one can choose the numbers $\underline{A}$, $\underline{B}$, $\underline{M}$, $\delta_0$ and $C^-$ (entering the definition of $W^-_\eps$; this can actually be chosen constant in $y$)
 small enough, and $\alpha_0$ large enough, so that
$$
\Lc \wme \leq 2 (\psi^0-\psi-\eps)~\text{in}~D^\eps ~~\text{and}~~\wme \leq w^\eps ~\text{on}~ \Gamma^\eps.
$$
As a consequence, with this choice of parameters, we have $\wme\leq w^\eps$ on $D^\eps$.
\end{itemize}
\end{propo}

\begin{proof}
\textbf{Uniform super-solution.} Let us start with the upper bound $\overline{M}$. Consider the function
$$
\varphi (\psi, y):= 1+\|w_0\|_\infty + \zeta_0 y,
$$
for some $\zeta_0>0$. Then $\varphi\geq w^\eps$ on $\Gamma^\eps$ provided $\eps$ is small enough, and 
$$
\Lc \varphi= \lambda_0 \nu \zeta_0.
$$
Hence $\Lc \varphi\geq 2(\psi^0-\psi-\eps)$ as soon as $\zeta_0\geq \sup_{y\in (0,Y)} {2 \psi^0}/({\lambda_0 \nu})$. Now, setting 
\be\label{def:barM}
\overline{M}:= 1+\|w_0\|_\infty +  Y\sup_{y\in (0,Y)} \frac{2 \psi^0}{\lambda_0 \nu},
\ee
we infer that $w^\eps\leq \overline{M}$ in $D^\eps$.
\gap

We now address the computations on the sub-solution $\wme$.

\textbf{Boundary constraints.} 
By construction, $\wme = w^\eps$ on the left- and right-hand boundaries, $\{\psi=0\}$ and $\{\psi = \psi^0(y)\}$. On $\{y=0\}$, close to $\psi=0$, we have
$$ w^\eps(\psi,0)=w_0(\eps+\psi)\simeq w_0(\eps) + w_0'(\eps) \psi + \Oc(\psi^2)= w_l^\eps(0) + w_0'(0)\psi + \Oc(\eps \psi+ \psi^2), $$
and therefore, for any choice of $\underline{A} > 0$, $0<\underline{B}<w_0'(0)$, there exists $\delta_0>0$ depending on $\underline{A},~ \underline{B}$ such that
 $\underline{A}\psi^{4/3}+\underline{B}\psi \leq w^\eps$ in $\{y=0~,~ 0<\psi<2\delta_0\}$
 for $\eps$ sufficiently small (we take the interval $[0,2\delta_0]$ to ensure a good overlap when transitioning into zone $II$).
 We now consider the boundary condition for $y=0$ and $\psi$ close to $\psi^0(0)$. There, we have
 $$ \wme (\psi,0)= C^-(0)^2(\psi^0(0)-\psi-\eps)^2 + (w_0(\psi^0-\eps)- C^-(0)^2 \eps^2), $$
 and there exists $\bar \psi \in [\psi+\eps, \psi^0(0)-\eps]\subset [\psi^0(0)-2\delta_0, \psi^0(0)]$ such that
\begin{eqnarray*}
\wme(\psi,0) - w^\eps(\psi,0)&=& C^-(0)^2 \left( \psi^0(0) - \psi -\eps\right)^2- C^-(0)^2\eps^2 \\&&+ w_0(\psi^0(0)-\eps) - w_0(\eps+\psi)\\
&=&\left(C^-(0)^2 -\frac{1}{2} w_0''(\bar \psi)\right) \left( \psi^0(0) - \psi -2\eps\right)^2 \\&&+ \left(2 \eps C^-(0)^2+ w_0'(\psi^0(0)-\eps)\right)  \left( \psi^0(0) - \psi -2\eps\right).
\end{eqnarray*}
The compatibility condition (\ref{compvmr}) entails that $w_0''(\psi^0(0))= 2 a(0)^2$, and therefore $w_0'(\psi)\sim -2a(0)^2 (\psi^0(0)-\psi)$ for $|\psi-\psi^0(0)|\ll1$.
 As a consequence, for $\eps$ and $\delta_0$ small enough, we have, if $C^-(0)<a(0)$ and $\psi>\psi^0(y)-2\delta_0$,
$$
C^-(0)^2 -\frac{1}{2} w_0''(\bar \psi)< 0,\quad 2 \eps C^-(0)^2+ w_0'(\psi^0(0)-\eps)< 0.
$$
Thus, $\wme(\psi,0) \leq  w^\eps(\psi,0)$ on $\{y=0, \psi^0(0)-\delta_0<\psi<\psi^0(0)-2\eps\}$.
 In the middle zone $\{y=0, \delta_0<\psi<\psi^0-\delta_0\}$, we have $\wme \leq w_\eps$ if we have, for example,
 $$ \underline{M} \leq \min_{\psi \in \left[\frac{\delta_0}{2},\psi^0(0)-\frac{\delta_0}{2}\right]} w_0(\psi), $$
 and $\eps$ small enough. This ensures that the total construction $\wme$ is smaller than $w^\eps_0$ everywhere on $\Gamma^\eps$.
\gap
 
In what follows, we choose $C^-$ small enough so that $\delta_0$ does not need to be changed, as per Proposition \ref{prop:rhb}. However, we are yet to set $\alpha_0$.

\textbf{Interior, zones $I$ and $II$.} We now need to choose the parameters so that
\begin{equation} \Lc\wme \leq  \Lc w^\eps = 2(\psi^0(y)-\psi-\eps). \label{goal} \end{equation}
in $D^\eps$. We have
\begin{equation} \begin{array}{rcl} \Lc\wme(\psi,y) & = &  h_I\left[\lambda_0\nu\derp{y}w_l^\eps(y)-\lambda_0\nu \alpha_0 (\underline{A}\psi^{4/3}+\underline{B}\psi)e^{-\alpha_0 y} -\frac{4}{9}\underline{A} \nu^2\sqrt{\wme}\psi^{-2/3} e^{-\alpha_0 y}\right] \\
& &-  h_{II} \lambda_0 \nu \alpha \underline{M} e^{-\alpha_0 y} \\
& & + h_{III} \Lc W^-_\eps + T \end{array} \label{lwme} \end{equation}
in which $T$, the ``transition term'', contains derivatives of the partition of unity, which are nonzero only in $\{\delta_0<\psi<2\delta_0\}$ and $\{\psi^0(y)-2\delta_0<\psi<\psi^0(y)-\delta_0\}$.

In zone $I$, we just have the first line of (\ref{lwme}), and the key term is the last one. 
More precisely, using the fact that
\begin{equation} \sqrt{\wme}\geq \sqrt{\frac{\tilde{B}}{2}}\psi^{1/2}e^{-\alpha_0 y/2} , \label{btilde} \end{equation}
for some $\tilde{B}\geq \underline{B}$,
 we get that for $0<\psi <2\delta_0$,
 \begin{multline}\label{est-zone1-1}
\lambda_0\nu(y)\derp{y}w_l^\eps(y)-\lambda_0 \nu(y)\alpha_0(\underline{A}\psi^{4/3}+\underline{B}\psi)e^{-\alpha_0 y} -\frac{4}{9}\underline{A} \nu(y)^2\sqrt{\wme}\psi^{-2/3} e^{-\alpha_0 y}\\
<- \nu(y)^2\frac{2\sqrt{2}}{9}\underline{A}\sqrt{\tilde{B}} e^{-3\alpha_0 y/2} \psi^{-1/6} -\frac{1}{2}\alpha_0 \nu(y)\lambda_0 \underline{B} \psi e^{- \alpha_0 y} + \Oc(\eps). 
 \end{multline}
The two first terms are negative, and therefore, for $\eps$ sufficiently small, the right-hand side of the inequality is negative on $[0, 2\delta_0]$.
\gap

Zone $II$ can be broken down into three parts: the first transition zone
$$ \omega_I=\{\delta_0<\psi<2\delta_0\} , $$
the main middle part where $h_{II}=1$, and the second transition zone,
$$ \omega_{III}=\{\psi^0(y)-2\delta_0<\psi<\psi^0(y)-\delta_0\} .$$

In the middle part, we just have $\Lc\wme =  -\lambda_0 \alpha_0 \nu(y) \underline{M} e^{-\alpha_0 y}\leq 0$, hence we can choose any $\alpha_0\geq 0$ for the moment,
 but this parameter will play an important role in the transition mechanism.
 \gap

\textbf{Transition area $\omega_I$.} In the area $\omega_I$, we have $h_I \neq 0$, so we should consider the first two lines of (\ref{lwme}) and the transition term $T$.
 Since the partition only depends on $\psi$ in $\omega_I$, we have that 
\begin{eqnarray*}
 T(\psi, y)&=&-\nu(y)^2 \sqrt{\wme} \p_\psi^2 h_I \left(w_l^\eps +\left(\underline{A} \psi^{4/3} + \underline{B} \psi - \underline{M} \right) e^{-\alpha_0 y}  \right)\\
 && \hspace{20pt} -\nu(y)^2 \sqrt{\wme} \left( 2 \p_\psi h_I \left(\frac{4}{3} \underline{A} \psi^{1/3} + \underline{B}\right)\right) e^{-\alpha_0 y}.
\end{eqnarray*}
As a consequence, we infer that there exists a constant $K_1= K_1(\underline{M}, \underline{A}, \underline{B})$ such that for $\eps$ sufficiently small,
$$
T(\psi, y)\leq K_1 e^{-3\alpha_0 y/2} \left(|\p_\psi h_I| + |\p_\psi^2 h_I|\right).
$$
Now, since $\p_\psi h_I$ and $\p_\psi^2 h_I$ vanish for $\psi=\delta_0$, by continuity and using (\ref{btilde}), there exist $\delta'>0$ such that for $\psi\in [\delta_0, \delta']$
 $$
 K_1\left(|\p_\psi h_I| + |\p_\psi^2 h_I|\right)\leq \left(\inf_{y\in Y}\nu(y)^2\right)\frac{1}{9\sqrt{2}}\underline{A}\sqrt{\tilde{B}} \psi^{-1/6}~~ \text{and}~~ h_I\geq 1/2.
 $$
 As a consequence, for $\psi\in [\delta_0, \delta_0 + \delta']$, we have,
$$
\Lc\wme<-\frac{\sqrt{2}}{9} \nu(y)^2 \underline{A}\sqrt{\tilde{B}} e^{ 3\alpha_0 y/2} \psi^{-1/6} - \frac{1}{2}\nu(y)\lambda_0 \alpha_0 \underline{B} \psi e^{ \alpha_0 y} + \Oc(\eps) <0.
 $$
 We now consider $\psi\in [\delta_0 + \delta', 2 \delta_0]$. On this set, there exists  positive constants $k_{\delta_0, \delta'},~K_{\delta_0, \delta'}$ such that $1-h_I\geq k_{\delta_0, \delta'}$, and 
 $$
 |\p_\psi h_I| + |\p_\psi^2 h_I|\leq K_{\delta_0, \delta'}.
 $$
 Hence, on this set, we have
 $$
 \Lc\wme<- k_{\delta_0, \delta'}  \lambda_0 \alpha_0 \nu(y) \underline{M} e^{-\alpha_0 y} + K_1 K_{\delta_0, \delta'}  e^{-3\alpha_0 y/2} .
 $$
 Thus we can choose $\alpha_0$ large (depending on all the other parameters), such that the right-hand side of the above inequality is negative for all $y\in [0,Y]$. 

So far, this is essentially the same construction and method as in \cite{OSbook}. It remains to look at zone $III$ and its transition zone $\omega_{III}$.
\gap

\textbf{Interior, zone $III$.} In zone $III$, choosing $C^-(y)$ small enough, we have, according to Proposition \ref{prop:rhb},
\be
\Lc \wme= \Lc W^-_\eps\leq  (2-\gamma) (\psi^0-\psi-\eps)
 \label{rightzone} \end{equation}
for some $\gamma\geq 1/2$.
\gap

\textbf{Transition area $\omega_{III}$.} Here, we observe a transition in the mechanism of being a sub-solution, going from one where the behaviour of $\Lc\wme$ is driven by
 the derivative of $e^{-\alpha_0 y}$, to one which is governed algebraically by the polynomial $P_y$.

Using the previous estimates, we have, since $h_{II}+ h_{III}=1$ in this area,
\begin{equation}
\begin{array}{rcl} \Lc\wme - 2 (\psi^0(y)-\psi - \eps) &\leq  & (1-h_{III}) \left( - \lambda_0 \nu(y)\alpha_0 \underline{M} e^{ -\alpha_0  y} - 2 \left(\psi^0-\psi-\eps\right)\right) \\
 & &-\gamma h_{III} (\psi^0 - \psi-\eps) + T, \end{array} \label{transition2-1}
\end{equation}
and the transition term $T$ is given in $\omega_{III}$ by
\begin{equation}
\begin{array}{rcl} T&=&\left( \nu(y) \lambda_0 \p_y h_{III} - \nu(y)^2 \sqrt{\wme} \p_\psi^2 h_{III} \right) ( W_\eps^- - \underline{M} e^{-\alpha_0 y})\\
&&- 2\nu(y)^2\sqrt{\wme} \p_\psi h_{III} \p_\psi W_\eps^-. \end{array} \label{transition}
 \end{equation}
Notice that in $\omega_{III}$,
$$
0\leq W_\eps^-\leq 4C^-(y)^2\delta_0^2 + \Oc(\eps^2),\quad |\p_\psi W_\eps^-|\leq 4C^-(y)^2 \delta_0,
$$
and therefore, choosing $C^-$ sufficiently small (depending on $\underline{M}$, $\alpha_0$ and $Y$), we can always require that
$$
W_\eps^-,~|\p_\psi W_\eps^-|\leq \underline{M} e^{-\alpha_0 y} ~~ \text{in}~~ \omega_{III}.
$$
Once again, we recall that choosing such a function $C^-$ has no impact on $\gamma$ or $\delta_0$. Thus there exists a constant $K_2$, depending only on $\lambda_0$, $\nu$ and $\underline{M}$, such that
\begin{equation}
|T(\psi, y)| \leq K_2  \left(| \p_y h_{III}| + | \p_\psi h_{III}| + | \p_\psi^2 h_{III}|\right) e^{-\alpha_0 y}.
\label{transition3} \end{equation}
 By continuity of $\p_y h_{III},~\p_\psi h_{III}$ and $\p_{\psi}^2 h_{III}$, we can pick $\delta''>0$ independently of $y$ so that, on the set
$$ \omega'' = \{ \psi^0(y)-\delta_0-\delta'' \leq \psi \leq \psi^0(y)-\delta_0 \} , $$
we have $h_{III}\geq 1/2$ and 
\begin{equation}
|T| \leq \frac{\gamma}{2} h_{III} (\psi^0 - \psi-\eps).
\label{tgamma} \end{equation}
Therefore, on the set $\omega''$, we have
 \begin{equation} \Lc\wme - 2(\psi^0(y)-\psi-\eps) \leq  0 . \label{32} \end{equation}
On the other hand, there exist $k_{\delta_0,\delta''},~K_{\delta_0, \delta''}>0$ such that for all $(\psi,y)\in \omega_{III}\backslash \omega''$, we have $1-h_{III}(\psi,y)\geq k_{\delta_0,\delta''}$ and
$$
|T| \leq K_{\delta_0, \delta''} e^{-\alpha_0 y}.
$$
So we have
$$
\Lc\wme - 2(\psi^0(y)-\psi-\eps) \leq \left[- \lambda_0 \nu \alpha_0 k_{\delta_0, \delta''} + K_{\delta_0, \delta''}\right] e^{-\alpha_0 y},
$$
and, choosing $\alpha_0>0$ large enough, depending on $\delta_0$ and $\delta''$, the right-hand side becomes negative. Estimate (\ref{32}) extends to all of $\omega_{III}$, thus it is proved on $D^\eps$.

%
%
%
%
%
\gap

\textbf{Conclusion.} The function $\wme$ satisfies the hypotheses of \lemref{comparl} (b) relative to $w^\eps$. As a result, the proposition is proved. 
\end{proof}

\subsubsection{Precise behaviour}

In this part, we aim to prove that $w^\eps(\psi,y)\sim a(y)^2(\psi^0(y)-\psi-\eps)^2$ near the right-hand boundary.
 To do so, we will take $\delta \leq \delta_0$, so that Proposition \ref{prop:blankets} holds.
Thanks to this property, we know that there exist positive parameters $M^-$, $M^+$, $A^-$ and $B^-$ (depending on $\delta$) such that
 $$ M^- \leq \wme(\psi,y) \leq w^\eps(\psi,y) \leq M^+ $$
 in zone $II$, and
 $$ w_l^\eps(y) + A^-\psi^{4/3} + B^-\psi \leq \wme(\psi,y) \leq w^\eps(\psi,y) \leq M^+ $$
 in zone $I$. These parameters depend on those in Proposition \ref{prop:blankets}, for instance
 $$ M^-  \leq \underline{M}e^{-\alpha_0 Y}. $$
 Some, such as $B^-\leq \underline{B}e^{-\alpha_0 Y}$ and $M^+=\overline{M}$, do not degrade with $\delta$. The goal now is to prove the precise behaviour in zone $III$.

\begin{propo} For any $\gamma>0$, there exist
\begin{itemize}
 \item functions $C^\pm(y) \in \Cc^1([0,Y])$ as in Proposition \ref{prop:blankets} with $\gamma\pm=\pm 2 \gamma$, and $\|a-C^\pm\|_{W^{1,\infty}} \leq K \gamma$ for some universal constant $K$;
 \item a positive number $\delta$ satisfying $\delta\leq \min(\delta_0,\delta^-,\delta^+)$, where $\delta_0$ is from Proposition \ref{prop:blankets} and $\delta^\pm$ are from Proposition \ref{prop:rhb},
 \item a positive parameter $\alpha$,
 \item and a positive number $y_1$ which depends on all of the above but not on $\eps$,
\end{itemize}
such that the functions $\Phi^\pm_\eps$ defined hereafter satisfy, for any $y_0\in [0, Y]$,
\begin{equation} \Lc \Phi^\pm_\eps \gtrless (2\pm\gamma)(\psi^0(y)-\psi-\eps) \label{goal2} \end{equation}
on the set
$$ D^\eps_{y_0,y_1} = \{(\psi,y)~|~y_0\leq y \leq \min(Y,y_0+y_1)~,~ 0 \leq \psi \leq \psi^0(y)-2\eps\} . $$

The functions are:
\begin{equation} \label{asdef}
\begin{array}{rcl} \Phi^+_\eps(\psi,y) & = & (h_I+h_{II})(\psi,y)M^+ e^{\alpha (y-y_0)}+ h_{III}(\psi,y) W^+_\eps \\
\Phi^-_\eps(\psi,y)  &= &h_I(\psi)\left[ w^\eps_l(y)+\left( A^-\psi^{4/3} + B^- \psi\right)e^{-\alpha (y-y_0)}\right]  \\
& & + h_{II}(\psi,y) M^- e^{-\alpha (y-y_0)}+ h_{III}(\psi,y) W^-_\eps , \end{array}\end{equation}
where the partition of unity $(h_I,h_{II},h_{III})$ is now adapted to the width $\delta$.

\label{altersols}
\end{propo}

\begin{proof} Propositions \ref{prop:blankets} and \ref{prop:rhb} prove that $\Lc \Phi^\pm_\eps$ satisfy the desired inequalities on zones $I$,
 $III$ and $II\backslash(\omega_{III})$, and that the functions $C^\pm$ can be chosen closer to $a$ as $\gamma$ becomes small (continuity of the polynomial $P_y$, combined with the fact that $P_y'(a)\neq 0$),
 so we only need to concentrate on the transition zone $\omega_{III}$. 
 \gap

\textbf{Sub-solutions.} The difference with the previous proof is that we can no longer use the smallness of $C^-$ to control the transition term $T^-$, given by (\ref{transition}).
 Instead of (\ref{transition3}), we have
 $$ |T| \leq K_2 (|\derp{y}h_{III}|+|\derp{\psi}h_{III}|+|\derp{\psi}^2 h_{III}|) (1+ e^{-\alpha (y-y_0)}) $$
 for a constant $K_2$ which also depends on $\delta$. While we can continue to use the smallness of the derivatives of $h_{III}$ in a $\delta''$-sized neighbourhood of the set $\{\psi = \psi^0(y)-\delta\}$,
 we are unable to control the sign of the quantity
 $$ -\lambda_0 \nu(y) \alpha M^- e^{-\alpha (y-y_0)}k_{\delta,\delta''} + K_{\delta,\delta''}K_2(1+e^{-\alpha (y-y_0)}) $$
 on $\{\psi^0-2\delta\leq \psi \leq \psi^0-\delta-\delta''\}$ by simply taking $\alpha$ large. Here, we have re-used the notations $k_{\delta,\delta''}$ and $K_{\delta,\delta''}$ from the previous proof.
 We need to assume that $e^{-\alpha (y-y_0)}$ is bounded from below, say, by $1/2$. Then, we have
 $$ -\lambda_0 \nu \alpha M^- e^{-\alpha (y-y_0)} h_{II} + |T| \leq -\frac{1}{2} \lambda_0 \nu \alpha M^- k_{\delta,\delta''} + 2 K_{\delta,\delta''}K_2, $$
 which is negative for $\alpha$ large enough. As a result, we see that this estimate is only valid locally in $y$, as long as $e^{-\alpha (y-y_0)}\geq 1/2$, hence on the set $D^\eps_{y_0,y_1}$ with $y_1 = \ln(2)/\alpha$.
 We stress, however, that the constants $K_2, K_{\delta,\delta''}, k_{\delta,\delta''}, M^-$ are global, i.e. do not depend on $y_0$. As a consequence, $\alpha$ and $y_1$ are uniform over the whole interval $[0,Y]$ and do not depend on $y_0$ either.
 \gap
 
\textbf{Super-solutions.} This time, in the transition zone $\omega_{III}$, we have the estimate
$$ |T| \leq K_2 (|\derp{y}h_{III}| + |\derp{\psi}h_{III}| + |\derp{\psi}^2 h_{III}|) (1+e^{\alpha (y-y_0)} + e^{3\alpha (y-y_0)/2}), $$
in which the power $3/2$ of the exponential comes from the nonlinearity (it was negligeable for subsolutions). This poses difficulties for all the estimates in the transition zone.
 On one hand, we need to be able to obtain smallness of $T$ in a $\delta''$-sized region where the derivatives of $h_{III}$ are small.
 We quickly see that $\delta''$ is going to depend on $\alpha$ if we try to work globally. On the other hand, we will want to prove that, outside of these small areas,
\begin{eqnarray*}
\lambda_0 \nu \alpha M^+ e^{\alpha (y-y_0)}h_{II} - |T|& \geq &\lambda_0 \nu \alpha M^+ e^{\alpha (y-y_0)}k_{\delta,\delta''} - K_{\delta,\delta''}K_2 (1+e^{\alpha (y-y_0)}+e^{3\alpha (y-y_0)/2}) \\&\geq& 2(\psi^0-\psi-\eps).
\end{eqnarray*}  
The negative term in the middle member is dominant for $\alpha$ large.

In both cases, bounding $T$ involves assuming that $e^{\alpha (y-y_0)}$ is bounded, for example by $2$. Then we can prove (\ref{tgamma}) again for a neighbourhood of the curve $\{\psi=\psi^0(y)-\delta\}$,
 and choose $\alpha$ large enough so that
 $$ \lambda_0 \nu \alpha M^+ e^{\alpha (y-y_0)}h_{II} - |T| \geq \lambda_0 \nu \alpha M^+ k_{\delta,\delta''} - 7 K_{\delta,\delta''} K_2 \geq 2\psi^0(y) . $$
 Once again, these estimates are valid for $y\in [y_0,y_0+y_1]$, with $y_1 = \ln(2)/\alpha$, with $\alpha$ independent of $y_0$.\end{proof}

\begin{coro}
Assume that a solution to problem \eqref{vma} exists. Then for any $\gamma>0$, there exist  a positive number $\delta>0$ and  functions $C^\pm(y) \in \Cc^1([0,Y])$ such that $C^+(y)>a(y)>C^-(y)$ and $\|a-C^\pm\|_{W^{1,\infty}} \leq K \gamma$ for some universal constant $K$ such that
\be\label{est:III}
W_\eps^-(\psi, y)\leq w^\eps(\psi, y) \leq W_\eps^+(\psi, y)\ee
for all $(\psi, y)$ in zone $III$.

\end{coro}
\begin{proof}
We argue by induction. Let $Y_k:=  k y_1$, $k\in \mathbb N$. Clearly estimate \eqref{est:III} is true for $y=Y_0=0$ and for $\delta$ small enough. As a consequence, choosing $y_0=Y_0=0$ in \eqref{asdef}, we infer that $\Phi^\pm_\eps \gtrless w^\eps$
 on the parabolic boundary of $D^\eps_{0,y_1}$. According to Proposition \ref{altersols} and Lemma \ref{comparl}, $\Phi^\pm_\eps \gtrless w^\eps$ on $D^\eps_{0,y_1}$., $\Phi^\pm_\eps \gtrless w^\eps$ on $D_{0, y_1}^\eps$, and in particular, estimate \eqref{est:III} holds in zone $III$ for all $y\in [Y_0, Y_1]$.

Now, assume that estimate \eqref{est:III} is true in zone $III$ for $y\in [0, Y_k]$ for some $k\geq 0$ such that $Y_k<Y$. Set $y_0=Y_k$ in \eqref{asdef}. By definition of $M^\pm, A^-, B^-$ and using our induction hypothesis, we have $\Phi^\pm_\eps \gtrless w^\eps$ on the parabolic boundary of $D^\eps_{Y_k,Y-Y_k}$. Therefore, according to Proposition \ref{altersols}, $\Phi^\pm_\eps \gtrless w^\eps$ on $D_{Y_k, y_1}^\eps$. It follows that estimate \eqref{est:III} is satisfied in zone $III$ for $y\in [Y_k, Y_{k+1}]$. By induction, we deduce that estimate \eqref{est:III} is true over the whole interval $[0,Y]$.

\end{proof}

\begin{rmk}
 We have obtained that, for every $\mu>0$, there exists $\delta>0$ such that for $\psi^0(y)-\delta \leq \psi \leq \psi^0(y)-2\eps$, we have
 $$ (a(y)^2\pm\mu)(\psi^0(y)-\psi-\eps)^2 \gtrless w^\eps(\psi,y), $$
 which, assuming convergence (which we prove later), proves the right-hand boundary behaviour $w(\psi,y) \sim a(y)^2(\psi^0(y)-\psi)^2$ as $\psi$ approaches $\psi^0(y)$.
\end{rmk}

\begin{lemma}[Existence of solutions] For every $Y>0$, there exists a classical solution to (\ref{vma}) on $D^\eps$, which is positive on $\overline{D^\eps}$.
\end{lemma}

\begin{proof}
Since we have shown above that a solution, if it exists, is positive, we can view (\ref{vma}) as a quasilinear parabolic equation
 with a strictly positive viscosity (for $\eps>0$ fixed) and apply classical results. This step is identical to \cite{OSbook} Lemma 2.1.7.
\end{proof}

Eventually, we have the following control from above close to the left-hand boundary:
\begin{lemma}
Let $w^\eps$ be the solution of \eqref{vma} on $D^\eps$. Then there  exist $B^+>0, A^+>0$  and $\delta>0$  that
$$
w^\eps(\psi, y)\leq w_l^\eps(y) + B^+ \psi - A^+ \psi^{4/3}
$$
for all $y\in [0,Y]$ and for all $\psi \in [0,\delta]$.\label{lem:sur-sol}
\end{lemma}
\begin{proof}
The proof is identical to part of the proof of Lemma 2.1.8 in \cite{OSbook}, and therefore we only recall the main arguments. We construct super-solutions of the form
$$ \Theta^+_\eps(\psi,y) = w^\eps(0,y) + (B^+\psi - A^+ \psi^{4/3})e^{\alpha y}, $$
with $B^+ > w_0'(0)$ and $A^+>0$, so that $\Theta^+_\eps(\psi,0)$ reaches the value $M^+$ at $\psi=\psi_1<\psi^0(0)/4$ and $\Theta^+_\eps(\psi)\geq w^\eps_0(\psi)$ for all $\psi\in[0,\psi_1]$.
 Then, in the same way as the blanket sub-solution was dealt with in zones $I$ and $II$, it is possible to choose $\alpha>0$ large enough so that $\Lc \Theta^+_\eps(\psi,y) \geq 2\psi^0(y)$,
 and so $\Theta^+_\eps \geq w^\eps$ on the set $\{0\leq y\leq Y~,~0 \leq \psi \leq \psi_1\}$, which yields a bound on the difference quotient and therefore on $p^\eps|_{\psi=0}$. Then, up to a renaming of the coefficients $A^+, B^+$, we have the desired result.
\end{proof}

\subsection{Lipschitz bounds}

We will now look for estimates on $p^\eps:=\derp{\psi}w^\eps$ and $q^\eps := \derp{y}w^\eps$. Bounds for these quantities are obtained in zones $I$ and $II$ in the same way as in \cite{OSbook},
 and by a similar argument to above, we isolate what goes on in zone $III$ and deal with it separately.

We start with the estimates on $p^\eps$:
\begin{propo}\label{lem:p-eps}
There exists a uniform constant $C$, depending only on $\lambda_0, \psi^0$, $\nu$ and $Y$, such that
$$
|\p_\psi w^\eps|\leq C\quad \text{in } \overline{D^\eps}.
$$
Furthermore, there exist $\delta>0$ and functions $E^\pm(y)>0$ such that for all $(\psi,y)$ such that $\psi^0(y)-\delta \leq \psi \leq \psi^0(y)-2\eps$,
\be -E^-(y)(\psi^0(y)-\psi-\eps) \leq  p^\eps(\psi,y)  \leq -E^+(y)(\psi^0(y)-\psi-\eps). \label{alterdpsi} \ee
\end{propo}

\begin{proof}
 The equation on $p^\eps$ is the following:
\begin{equation} \left\{ \begin{array}{rcl} \lambda_0 \nu(y)\derp{y}p^\eps - \frac{\nu(y)^2}{2\sqrt{w^\eps}}p^\eps\derp{\psi}p^\eps - \nu(y)^2\sqrt{w^\eps}\derp{\psi}^2 p^\eps & = & -2 \\
p^\eps|_{y=0} & = & (w^\eps_0)'(\eps+\psi) . \end{array} \right. \label{dpsiw} \end{equation}

The sub- and super-solutions constructed in Propositions \ref{prop:blankets} and \ref{altersols}  and Lemma \ref{lem:sur-sol} ensure that
$$\ba
 0< B^- \leq p^\eps_{|\psi=0}\leq B^+,\\
-2 C^+(y)^2 \eps\leq  p^\eps_{|\psi=\psi^0(y)-2\eps}\leq -2 C^-(y)^2 \eps.\ea
$$

Therefore there exists a uniform constant $K_0$ (depending on $B^\pm$, $C^\pm$, $w_0$ and $Y$) such that $|p^\eps|\leq K_0$ on $\Gamma^\eps$.
 Then it is easily checked that $K_0$ (resp. $-K_0-2y/\lambda_0$) is a super- (resp. a sub-) solution of \eqref{dpsiw}. Therefore, according to Lemma \ref{comparl}, which does apply to the operator $\Mc$
 associated with the equation, which is
 $$ \Mc f = \lambda_0 \nu(y)\derp{y}f - \frac{\nu(y)^2}{2\sqrt{w^\eps}} f\derp{\psi}f - \nu(y)^2\sqrt{w^\eps} \derp{\psi}^2 f , $$
we have
\begin{equation}
-K_0 - \frac{2 y}{\lambda_0}\leq  p^\eps\leq K_0 ~~ \text{in}~~ D^\eps,
\label{blanketp} \end{equation}
and the uniform bound is proved.
\gap

It only remains to prove the refined bound close to the right-hand boundary. We use the same method as in Proposition \ref{altersols},
 with the only difference that we do not need to distinguish between zones $I$ and $II$. More precisely, we define
 $$ 
  \pi^\pm(\psi,y) := \pm(1-h_{III})\left(K_0+\frac{2Y}{\lambda_0}\right) e^{\zeta y} - h_{III}E^\pm(y)(\psi^0(y)-\psi-\eps) ,
$$
for $\zeta>0$ sufficiently large and for some positive functions $E^\pm$ to be chosen so that 
 $\pi^\pm(\psi,y) \gtrless p^\eps$ on $\Gamma^\eps$, and $\Mc\pi^\pm \gtrless -2$.
The inequalities $\pi^\pm(\psi,y) \gtrless p^\eps$ are satisfied on the right-hand boundary $\{\psi=\psi^0-2\eps\}$ as soon as
$0<E^+(y)<2 C_-(y)^2$ and $E^-(y)>2 C_+(y)^2$. We now compute $\Mc \pi^\pm$ to get the other constraint: setting $C_0:= K_0+{2Y}/{\lambda_0}$, we have
\begin{eqnarray}
\nonumber\Mc \pi^\pm&=& \pm \zeta \lambda_0\nu C_0(1-h_{III})  e^{\zeta y} \\&&+ h_{III}\left[ -E^\pm \lambda_0\nu(\psi^0)' - \frac{\pi^\pm\nu^2}{2\sqrt{w^\eps}} E^\pm - \lambda_0 \nu(E^\pm)' (\psi^0-\psi-\eps)\right] + T'\\
\nonumber&=& \pm \zeta \lambda_0 \nu C_0(1-h_{III}) e^{\zeta y} - h_{III}(1-h_{III})\frac{\pm C_0  e^{\zeta y}\nu^2}{\sqrt{w^\eps}} E^\pm \nonumber\\
&&\nonumber+ \frac{\nu^2}{2\sqrt{w^\eps}} h_{III} (1-h_{III}) E^\pm(y)^2 (\psi^0-\psi-\eps)+ T'\\
\label{mpi}&&+ h_{III}\left[-E^\pm \nu\lambda_0(\psi^0)' + \nu^2\frac{\psi^0-\psi-\eps}{2\sqrt{w^\eps}} (E^\pm )^2- \lambda_0 \nu(E^\pm)' (\psi^0-\psi-\eps) \right] ,
\end{eqnarray}
where the term $T'$ is, as in the proof of Proposition \ref{altersols}, a transition term involving derivatives of $h_{III}$. We first focus on the last line. 
Using Proposition \ref{altersols}, we note that $w_r^\eps(y)\sim a(y)\eps^2 \lessgtr C^\pm(y)^2 \eps^2$, so that
$\sqrt{w^\eps}$ is surrounded by $C^\pm(\psi^0-\psi-\eps)$, and thus $\frac{\psi^0-\psi-\eps}{2\sqrt{w^\eps}}$ is $\Oc(1)$.
 The third term inside the brackets in \eqref{mpi} is of lower order, and can be dealt with by taking some margin on $E^\pm$ and eventually adjusting $\delta$.
 Thus we leave it aside for the time being and we will explain the details later.

According to  Proposition \ref{altersols}, we have
$$
-E^\pm \lambda_0 \nu(\psi^0)' + \nu^2\frac{\psi^0-\psi-\eps}{2\sqrt{w^\eps}} (E^\pm)^2 \gtrless  -E^\pm \lambda_0\nu (\psi^0)' + \frac{\nu^2}{2C^\pm}(E^\pm)^2 ,
$$
hence we seek solutions of the inequations
\begin{equation} -E^\pm \lambda_0 \nu (\psi^0)' + \frac{\nu^2}{2C^\pm}(E^\pm)^2 +2 \gtrless \pm \gamma\label{fpnd} \end{equation} for some $\gamma>0$.
Given that $C^\pm(y)$ are close to $a(y)$, let us focus on the polynomial
$$ Q_y(E) = \frac{\nu(y)^2}{2a(y)}E^2 -\lambda_0 \nu(y)(\psi^0)'(y) E + 2. $$
We would like $Q_y(E^\pm) \gtrless \pm \gamma$. As $Q_y'(0)<0$, we can choose $E^+(y)$ positive and small, for instance. As for $E^-$, we make the following observations.
 First, notice that $2a(y)^2$ is always a root of $Q_y$. Hence the other root of $Q_y$ is $2/(a(y)\nu(y)^2)$.
 We want $E^-(y)$ to belong to the interval whose bounds are $2a(y)^2$ and $2/(a(y)\nu(y)^2)$, so that $Q_y(E^-(y))<0$. 
 This is possible and  coherent with the boundary constraint $E^-(y)>2 a(y)^2$ if and only if $2a(y)^2 < 2/(a(y)\nu(y)^2)$, i.e. $a(y)<\nu(y)^{-2/3}$. Noticing that
$$
P_y(\nu(y)^{-2/3})=2-\nu(y)^{-1/3}\lambda_0 (\psi^0)'(y),
$$
it suffices that $P_y(\nu^{-2/3})>0$ to have the desired inequality between $a$ and $\nu$, given the variations of $P_y$. This yields the condition
$$
\frac{\lambda_0 (\psi^0)'(y)}{\nu(y)^{1/3}}>2,
$$
which is the one we announced in the statement of \thref{exist}. Hence, if this is satisfied, we first choose the continuous function $C^+(y)$ in Proposition \ref{altersols} so that
$$ a(y)^2<C^+(y)^2<1/(a(y)\nu(y)^2) $$
for all $y$, and then, by continuity, we can choose $E^-(y)$ so that $2 (C^+)^2<E^-<2/(a\nu^2)$. Thus \eqref{fpnd} is satisfied for $E^-$, provided $C^-$ is chosen sufficiently close to $a$.
\gap

At this stage, we have proved that
\begin{equation}
\begin{array}{rcl} \Mc \pi^\pm &\gtrless &  \pm \zeta \lambda_0 C_0 \nu (1-h_{III}) e^{\zeta y}  - h_{III}(1-h_{III})\frac{\pm C_0  \nu^2 e^{\zeta y}}{\sqrt{w^\eps}} E^\pm \\&&+ \frac{\nu^2}{2\sqrt{w^\eps}} h_{III} (1-h_{III}) E^\pm(y)^2 (\psi^0-\psi-\eps)\\
&&+ h_{III}\left[-2 \pm \gamma - \lambda_0 \nu (E^\pm)' (\psi^0-\psi-\eps) \right] + T'. \end{array} \label{mpi2}
\end{equation}
 Now, choosing $\delta$ sufficiently small (depending on $\gamma$), we have
 $$
 \left| \lambda_0 \nu (E^\pm)' (\psi^0-\psi-\eps)\right| \leq \frac{\gamma}{2},\quad \forall (\psi,y) \in \{\psi^0(y)-2\delta, \psi^0(y)-2\eps\}.
 $$
 Furthermore, provided $\zeta$ is sufficiently large (depending on $\delta$, $E^\pm$ and $C^\pm$), the first term on the first line of (\ref{mpi2}) can absorb the others (bar $T'$) and still be
 $\gtrless (-2 \pm \gamma/2) (1-h_{III})$. The transition term $T'$ is treated exactly as $T$ was in the zone $\omega_{III}$ in Proposition \ref{altersols}.
 With $\delta$ now fixed, we will obtain a control of $T'$ which is local in $y$, and we notice that we can cover $D^\eps$ with a finite number of these local estimates thanks to the global bounds on $p^\eps$
 in (\ref{blanketp}). Since the arguments are identical to the ones of Proposition \ref{altersols}, we leave them to the reader.
\end{proof}

\begin{rmk}
Notice that 
$$
E^+(y)<2 C^-(y)^2 < 2 a(y)^2 < 2 C^+(y)^2 < E^-(y),
$$
and that one can take for instance $C^\pm(y)= a(y)\pm \mu$, $E^\pm(y)= 2(a(y)\pm 2 \mu)^2$, for any $\mu>0$ small enough. As emphasised before, the constant $\delta$ then depends on $\mu$.

\end{rmk}

We now address the estimates on $q^\eps=\p_y w^\eps$. We first give uniform estimates in zones $I$ and $II$, which are proved in \cite{OSbook}. We briefly recall the main arguments of the proof.

\begin{lemma}[\cite{OSbook}, Lemmata 2.1.8 to 2.1.13]
Let $\delta>0$ be arbitrary.  The function  $q^\eps$ is uniformly bounded in $\eps$ on
 $$ \overline{\Delta_\delta} := \{(\psi,y)~|~ 0 \leq y \leq Y \mathand{and} 0 \leq \psi \leq \psi^0(y)-\delta\},
  $$
  and there exists $\beta\in (0,1/2)$ such that $(w^\eps)^{\beta-1}\p_y w^\eps$ is bounded uniformly in $\eps $ on $\overline{\Delta_\delta}$. 
  
 \label{lipzone12}
\end{lemma}

\begin{proof}
First, standard parabolic regularity estimates from \cite{Fabook} yield uniform bounds for $q^\eps$ in zone $II$, since in that zone $w^\eps$ is bounded away from zero uniformly in $\eps$
 and uniformly Lipschitz continuous with respect to $\psi$ according to Proposition \ref{lem:p-eps}. Hence it remains to look at a neighbourhood of the left-hand boundary.
 Using equation \eqref{vm}, we find that the equation on $q^\eps = \derp{y}w^\eps$ is
 \begin{multline} \lambda_0 \nu\derp{y}q^\eps -2 \lambda_0 \chi' \chi'' q^\eps+ \frac{1}{w^\eps}\left((\psi^0(y)-\psi-\eps)q^\eps - \frac{\lambda_0\nu}{2}(q^\eps)^2\right) - \nu^2\sqrt{w^\eps}\derp{\psi}^2 q^\eps \\
 = 2(\psi^0)'(y) - \frac{8\chi'\chi''}{\nu} (\psi^0-\psi-\eps). \label{dyw} \end{multline}
It can be easily checked that this equation has a minimum principle, in the sense that if $q^\eps$ has a negative minimum in $D_\delta$, then it is possible to bound its value from below.
 Furthermore, on the left-hand boundary $\psi=0$, we have
$$
q^\eps_{|\psi=0}=\p_y w_l^\eps(y)= \mu(\eps)e^{\mu(\eps) y/w_0(\eps)}= \Oc(\eps),
$$
and therefore $q^\eps$ is bounded on the left-hand boundary, which  leads to an overall lower bound for $q^\eps$.
 However, the nature of the nonlinearity in the equation on $q^\eps$ makes a similar maximum principle impossible to prove (we will make this remark again shortly).
 Ole\u{i}nik and Samokhin then use a nonlinear change of unknown to make a maximum principle appear on a modified equation, allowing them to bound $\sqrt{w^\eps}\derp{\psi}^2 w^\eps$,
 and therefore $q^\eps$ by simply using the equation. They finally show a bound for $(w^\eps)^{\beta-1}q^\eps$ for $\beta\in (0,1/2)$, which we saw the usefulness of in the proof of \cororef{equivcoro}.
\end{proof} 

We now seek the explicit and precise behaviour of $q^\eps$ in zone $III$. As it was shown in \propref{altersols} that $w^\eps$ would vanish quadratically at $\psi=\psi^0(y)-\eps$,
 we expect $q^\eps$ to decay linearly. We prove the following.
 
\begin{propo}
There exists  a constant $C_0>0$ such that
\be | q^\eps(\psi,y)|  \leq C_0(\psi^0(y)-\psi-\eps), \label{alterdy}
\ee
for all $(\psi,y)\in D^\eps$ such that $\psi\geq \psi^0(y)-\delta$.
 \label{alterd}
\end{propo}

\begin{proof}
 As mentioned before, if $q^\eps$ has an interior minimum $q_\text{min}^\eps$, then using \eqref{dyw} we find that
 $$
-2 \lambda_0 \chi' \chi''q_\text{min}^\eps + \frac{1}{w^\eps}\left((\psi^0(y)-\psi-\eps)q^\eps_\text{min} - \frac{\lambda_0\nu}{2}(q^\eps_\text{min})^2\right) \geq -C,
 $$
 for some constant $C$, and therefore $q_\text{min}^\eps\geq -\bar Q$, for some constant $\bar Q$ depending only on the data, and independent of $\eps$. 
 However, it is impossible to do the same for an interior maximum,
 due to the zero-order nonlinearity, in which, when one examines a positive maximum, $q^\eps$ and $-(q^\eps)^2$ have opposite signs.
 One way around this is to make a nonlinear change of unknown. In \cite{OSbook} Lemma 2.1.12, Ole\u{i}nik and Samokhin set $\derp{\psi}w = \vphi(S)$ (the function $\vphi$ is essentially exponential),
 and examine $\sqrt{w}\derp{\psi}S$, which gives them a bound on $\sqrt{w}\derp{\psi}^2w = \derp{y}w$ according to the equation. However, we are looking for more than just a bound on $q^\eps$,
 we are seeking a more precise behaviour in zone $III$.
 Rather, we consider the unknown
 $$ r^\eps = -\frac{q^\eps}{p^\eps} = -\frac{\derp{y}w^\eps}{\derp{\psi}w^\eps}. $$
 We can do this because we have inequality (\ref{alterdpsi}), which implies that $-p^\eps >0$ in zone $III$. We will show that $r^\eps$ is uniformly bounded,
  which will prove that $q^\eps = \Oc (\psi^0(y)-\psi-\eps)$ as $\psi$ goes to $\psi^0(y)-2\eps$.
 
Using equations (\ref{dpsiw}) and (\ref{dyw}), we get
 \begin{eqnarray} \lambda_0 \nu\derp{y}r^\eps &=& \frac{2}{p^\eps}(r^\eps-(\psi^0)'(y)) - \frac{\nu^2\sqrt{w^\eps}}{p^\eps}(\derp{\psi}^2 q^\eps + r^\eps \derp{\psi}^2 p^\eps)\label{dyr}\\
 &&+ 2 \lambda_0 \chi' \chi'' r^\eps + 8 \frac{\chi' \chi''}{\nu p^\eps} \left( \psi^0(y)-\psi-\eps\right)\nonumber\\
 &&+ \frac{r^\eps}{2w^\eps}\left[ -2(\psi^0(y)-\psi-\eps) + {\lambda_0\nu } q^\eps - \nu^2\sqrt{w^\eps} \p_\psi p^\eps \right]\label{dyr2}.
  \end{eqnarray}
   Remember that $p^\eps=\derp{\psi}w^\eps$, so, by using (\ref{vma}),
    \be\label{identity-pq} \nu^2\sqrt{w^\eps}\derp{\psi}p^\eps = \nu^2\sqrt{w^\eps}\derp{\psi}^2 w^\eps = \lambda_0 \nu q^\eps - 2(\psi^0(y)-\psi-\eps). \ee
  Therefore the term in brackets in \eqref{dyr2} is zero.
  
 Then, differentiating $r$ with respect to $\psi$ gives us the following equalities:
 \begin{eqnarray*} \derp{\psi}r^\eps & = & \frac{-1}{p^\eps} (\derp{\psi}q^\eps + r^\eps \derp{\psi}p^\eps) \\
  \sqrt{w} \derp{\psi}^2 r^\eps & = & -\frac{\sqrt{w^\eps}}{p^\eps}(\derp{\psi}^2 q^\eps + r^\eps \derp{\psi}^2 p^\eps) + 2\sqrt{w^\eps}\frac{\derp{\psi}p^\eps}{(p^\eps)^2}(\derp{\psi}q^\eps + r^\eps \derp{\psi} p^\eps) \\
  & = &  -\frac{\sqrt{w^\eps}}{p^\eps}(\derp{\psi}^2 q^\eps + r^\eps \derp{\psi}^2 p^\eps) - 2\sqrt{w^\eps} \frac{\derp{\psi}p^\eps}{p^\eps}\derp{\psi} r^\eps.
 \end{eqnarray*}
 We recognise the first term of $\sqrt{w^\eps}\derp{\psi}^2 r^\eps$ in (\ref{dyr}), and we once again use \eqref{identity-pq} to interpret $\derp{\psi}p^\eps$ in the second term.
   As a result, we have the following equation on $r^\eps$:
  \begin{eqnarray}
  \lambda_0 \nu \derp{y}r^\eps &=& \nu^2\sqrt{w^\eps}\derp{\psi}^2 r^\eps + \frac{2}{p^\eps}\left[r^\eps-(\psi^0)'(y)\right] + 2\lambda_0\nu r^\eps \derp{\psi} r^\eps - \frac{4(\psi^0(y)-\psi-\eps)}{p^\eps} \derp{\psi}r^\eps\nonumber\\ \label{reqn}
  &&+ 2 \lambda_0 \chi' \chi'' r^\eps + 8 \frac{\chi' \chi''}{\nu p^\eps} \left( \psi^0(y)-\psi-\eps\right).
  \end{eqnarray}
 This equation has a maximum and a minimum principle in zone $III$. Indeed, let us consider $(\psi_m,y_m)$ a minimum of $r^\eps$ in zone $III$ deprived from  its parabolic boundary. Then we have $\derp{y}r^\eps\leq 0$, $\derp{\psi}r^\eps = 0$,
  $\sqrt{w^\eps}\derp{\psi}^2 r^\eps \geq 0$, and therefore
 $$ \frac{2}{p^\eps}\left[r^\eps - (\psi^0)'(y_m)\right]+ 2 \lambda_0 \chi' \chi'' r^\eps + 8 \frac{\chi' \chi''}{\nu p^\eps} \left( \psi^0(y_m)-\psi_m-\eps\right) \leq 0 .$$ 
 We can always choose $\delta$ so that $1/|p^\eps|\geq |2 \lambda_0 \chi' \chi'' |$ in zone $III$, and therefore the second term is absorbed by the first one.
 The last term is bounded according to Proposition \ref{lem:p-eps}, and therefore we obtain a lower bound on $r^\eps$ at a point of interior minimum since $p^\eps$ is negative:
 \begin{eqnarray*} r^\eps(\psi_m,y_m) & \geq & 2(\psi^0)'(y_m) - 8\frac{|\chi'(y_m)\chi''(y_m)|}{\nu(y_m)} (\psi^0(y_m)-\psi_m-\eps) \\
   & \geq & \inf_{y\in [0, Y]}\left[ 2(\psi^0)'(y) - 8\frac{|\chi'(y)\chi''(y)|}{\nu(y)} \psi^0(y)\right].
 \end{eqnarray*}
 
 Likewise, there is an upper bound on the values of interior maxima: in the same way, at a maximum point $(\psi_M,y_M)$, we have
 $$ \frac{2}{p^\eps}\left[r^\eps - (\psi^0)'(y_M)\right]+ 2 \lambda_0 \chi' \chi'' r^\eps + 8 \frac{\chi' \chi''}{\nu p^\eps} \left( \psi^0(y_M)-\psi_M-\eps\right) \geq 0, $$
 $$ \mathrm{so}~~ r^\eps(\psi_M,y_M) \leq \max_{y\in[0,Y]} \left(2 (\psi^0)'(y) + 8\delta \frac{|\chi'(y)\chi''(y)|}{1+\chi'(y)^2} \right). $$
 
 It remains to examine the boundary conditions for $r^\eps$ in zone $III$. Thanks to the compatibility conditions \eqref{compvml} and \eqref{compvmr}, equation \eqref{vma} is valid up to the boundary, and we have
 $$ \lambda_0 \nu(0) q^\eps|_{y=0} = \nu(0)^2\sqrt{w^\eps_0}(w^\eps_0)'' + 2(\psi^0(0)-\psi-\eps), $$
 which, given the compatibility condition (\ref{compvmr}), becomes, as $\psi \to \psi^0(0)-2\eps$,
 $$ \lambda_0 \nu(0) q^\eps|_{y=0} \sim 2(\nu(0)^2a(0)^3+1)(\psi^0(0)-\psi-\eps) = 2 \lambda_0 \nu(0) (\psi^0)'(0) a(0)^2 (\psi^0(0)-\psi-\eps),  $$
 since $a(0)$ solves $P_0(a(0))=0$. As a result, because of the initial condition on $p^\eps$, seen in (\ref{dpsiw}), $r^\eps_{|y=0}$ is bounded in zone $III$.

 To get the boundary condition on $\{\psi=\psi^0(y) -2\eps\}$, we differentiate with respect to $y$ the identity
 $$
 w^\eps(\psi^0(y)-2\eps,y)= w_0(\psi^0(0)-\eps) \frac{a(y)^2}{a(0)^2},\quad \forall y\in (0,Y).
 $$
 We get the relation
 $$ 2a'(y)a(y)\frac{w_0(\psi^0(0)-\eps)}{a(0)^2} = [(\psi^0)'(y)p^\eps + q^\eps]|_{\psi = \psi^0(y)-2\eps}. $$
 We proved that $-E_-(y)\eps\leq p^\eps_{\psi=\psi^0-2\eps} \leq -E_+(y)\eps$ on the boundary, while $w_0(\psi^0(0)-\eps)= O(\eps^2)$, so the left-hand side is of lower order. We obtain
 $$ r^\eps|_{\psi = \psi^0(y)-2\eps} = (\psi^0)'(y) + \Oc(\eps). $$
 Finally, we know by the Schauder estimates that $r^\eps$ is uniformly bounded along the curve $\{\psi = \psi^0(y)-\delta\}$.
 
 Combining the interior extremum bounds and the boundary conditions, we get that $r^\eps$ is uniformly bounded in zone $III$, which ends the proof of (\ref{alterdy}) and \propref{alterd}. 
 
\end{proof}

\section{Solution to the transformed equation and uniqueness}
\label{ssec:conclu}

We come to the final stage of the proof of Proposition \ref{prop:exvm}. We first tackle the existence: for every set $K$ that is compact in $D$, which we recall is
$$ D = \{(\psi,y)\in\Rr^2 ~|~ 0<y<Y \mathand{and} 0< \psi < \psi^0(y) \}, $$
 we have uniform Lipschitz bounds for the family of solutions $(w^\eps)_{\eps\leq \eps_0}$,
 so we can extract a subsequence that converges uniformly to $w$ by the Ascoli theorem. Moreover, the parabolic regularity gives us
 convergence of the derivatives on $K$. The behaviour of the sub- and super-solutions $\Phi^\pm_\eps$ yield continuity of $w$ and its
 first derivatives on $\overline{D}$, so $w$ vanishes linearly at $\psi=0$ with a positive derivative there, as per the blanket sub-solution,
 and quadratically along the curve $\{\psi=\psi^0(y)\}$, with the precise behaviour being shown in Proposition \ref{altersols}.
\gap

It remains to show the uniqueness of classical solutions of \eqref{vm} with the properties listed in Proposition \ref{prop:exvm}. The arguments are similar to Theorem 2.1.15 in \cite{OSbook}. In fact, we prove uniqueness in a slightly larger class, namely
\begin{propo}\label{prop:uniqueness-vm}
Assume that $w_1, w_2$ are two classical solutions of \eqref{vm} such that for $i=1,2$
$$
\begin{aligned}
\inf_{y\in [0,Y]}\p_\psi w_i\vert_{\psi=0}>0,\\
w_i(\psi, y)\sim a(y) (\psi^0(y)-\psi)^2 \text{ for }\psi\text{ close to } \psi^0(y),\\
\|\sqrt{w_i} \p_{\psi\psi } w_i\|_\infty\leq C_0
\end{aligned}
$$
for some positive constant $C_0$. Then $w_1=w_2$.

\end{propo}

\begin{proof}
We  have
\begin{equation} \lambda_0\nu\derp{y}(w_1-w_2) - \nu^2\sqrt{w_1}\derp{\psi}^2(w_1-w_2) - \frac{\nu^2}{\sqrt{w_1}+\sqrt{w_2}}\derp{\psi}^2w_2 (w_1-w_2) = 0 . \label{diffeq} \end{equation}
Let us first derive some bounds on
$$ \frac{1}{\sqrt{w_1}+\sqrt{w_2}} \derp{\psi}^2 w_2 (w_1-w_2)= (\sqrt{w_1}- \sqrt{w_2})\derp{\psi}^2 w_2 $$
near the boundaries. We first notice that for each $i$, $\sqrt{w_i} \p_\psi^2 w_i$ is bounded in $D$ by assumption.
 Furthermore, according to the inequalities satisfied by $w_1$ and $w_2$ near the left-hand and right-hand boundaries, it is easily checked that $w_1/w_2$ is also bounded (and bounded away from zero) in the vicinity of these boundaries.
 The boundedness of $(\sqrt{w_1}-\sqrt{w_2})\derp{\psi}^2 w_2$ follows.
 
However, since $(\sqrt{w_1}+\sqrt{w_2})^{-1}\derp{\psi}^2 w_2$ is not expected to be bounded near the boundaries, we cannot conclude immediately with a Gr\"onwall-type argument.
 To reduce to a case where this coefficient would be bounded, we consider the linear operator
\begin{eqnarray*} \Nc f & = &  \lambda_0\nu\derp{y}f - \nu^2\sqrt{w_1}\derp{\psi}^2 f - F_\tau f, \\
\mathand{where} F_\tau(\psi,y) & = & \left\{ \begin{array}{ll} \frac{\nu^2}{\sqrt{w_1}+\sqrt{w_2}}\derp{\psi}^2 w_2 & \mathand{if} \psi\geq \tau \mathand{and} \psi\leq \psi^0(y)-\tau \\
0 & \mathand{otherwise,} \end{array} \right. \end{eqnarray*}
with $\tau$ to be chosen shortly. Let us define a crude super-solution of the equation $\Nc f=0$, namely
$$ \tilde{\Phi}^+ (\psi,y) = e^{\alpha y}\Big[h_{I} (A^+\psi^{4/3}+B^+\psi) + h_{II} M^+ + h_{III} (\psi^0-\psi)^{5/6}\Big], $$
with $A^+$, $B^+$, $M^+$ and $\delta$ so that the function in large brackets is larger than $w^\eps$.
 This is crude in the sense that its decay at $\psi=\psi^0(y)$ is slower than $\Oc((\psi^0-\psi)^2)$. Using the sub- and super-solutions from \propref{altersols} that surround $w_1$, we see that
$$ \Nc \tilde{\Phi}^+ \geq k\psi^{-1/6} \mathand{and} \Nc \tilde{\Phi}^+ \geq k(\psi^0-\psi)^{-1/6} $$
in zones $I$ and $III$ respectively (adjust $\delta$ if needed), and then choose $\alpha$ large enough to ensure $\Nc \tilde{\Phi}^+> 0$ in zone $II$.
 Because of transition terms, we may have to act locally in $y$, but by choosing the parameters large enough, we can re-use the ``blanket'' strategy to cover all of $D$.
\gap

Let $\mu>0$, and consider the functions
$$ W^\pm = \mu\tilde{\Phi}^+ \pm (w_1 - w_2). $$
We will show that both of these functions are non-negative, hence, for any $\mu>0$,
$$ |w_1-w_2| \leq \mu\tilde{\Phi}^+, $$
proving that $w_1=w_2$.

We have
\begin{eqnarray*} \Nc W^\pm & = & \mu \Nc \tilde{\Phi}^+ \pm \Nc (w_1-w_2) \\
 & = & \left\{ \begin{array}{ll} \mu \Nc \tilde{\Phi}^+ & \mathand{if} \tau \leq \psi \leq \psi^0(y)-\tau \\
 \mu \Nc \tilde{\Phi}^+ \pm \frac{\nu^2}{\sqrt{w_1}+\sqrt{w_2}} \derp{\psi}^2 w_2 (w_1-w_2) & \mathand{otherwise} \end{array} \right. \\
 & \geq & \left\{ \begin{array}{ll} 0 & \mathand{if} \tau\leq \psi \leq \psi^0(y)-\tau \\
 \mu k \psi^{-1/6} - m & \mathand{if} 0\leq \psi \leq \tau \\
 \mu k (\psi^0(y)-\psi)^{-1/6} - m & \mathand{if} \psi^0(y)-\tau \leq \psi \leq \psi^{0}(y), \end{array} \right. \\ & \geq & 0 \end{eqnarray*}
for some positive constant $m$ independent of the parameters, if $\tau$ is small enough ($\tau$ depends on $\mu$).

If $w_1\neq w_2$ at a certain point, then, at that point, either $W^+$ or $W^-$ is negative for $\mu$ small enough. First of all, we note that it must be an interior point of $D$,
 since $w_1=w_2=0$ on the boundary, and $\tilde{\Phi}^+$ is zero on the left- and right-hand boundaries, and is positive at $y=0$.
 Next, $W^\pm$ can be negative if and only if $R^\pm = W^\pm e^{\kappa y}$ can be, for any $\kappa >0$. We quickly compute that
\begin{eqnarray*} \Nc W^\pm & = & e^{-\kappa y} \Big[ \Nc R^\pm - \lambda_0 \nu \kappa R^\pm\Big] \\
 & = & e^{-\kappa y} \Big[ \lambda_0 \nu \derp{y}R^\pm - \nu^2 \sqrt{w_1} \derp{\psi}^2 R^\pm - (F_\tau + \lambda_0 \nu \kappa) R^\pm \Big], \end{eqnarray*}
and recall that this quantity is positive. But, if $R^\pm$ has a negative minimum, and $\lambda_0 \nu \kappa$ is larger than $\max |F_\tau|$, we see that $\Nc W^\pm \leq 0$ at that minimum.
 So $R^\pm$ cannot be negative, and the uniqueness for problem (\ref{vm}) is proved. $\square$

\end{proof}

We then infer uniqueness for the Prandtl equation in an appropriate class.
\begin{coro}\label{cor:uniqueness}
Consider a Lipschitz solution $(u,v)$ of \eqref{pr} such that $\p_{\xi \xi} v$ is bounded in $[0,Y]\times \mathbb R_+$, and such that there exist $\xi_0>0$, $m,M>0$ such that
$$
m\xi\leq v(\xi, y)\leq M \xi\quad \forall \xi \in [0,\xi_0],\ \forall y\in [0,Y]
$$
and
$$
v(\xi,y)\sim a(y) (\psi^0(y)-\xi)\quad\text{as }\xi\to \infty.
$$

Then $(u,v)$ is unique.
\end{coro}

\begin{proof}
We can see in our computations in section \ref{ssec:transfo} that there is an equivalence
 between solutions $v$ of the Prandtl equation (\ref{full-pr}) that are Lipschitz-class, with a positive derivative with respect to $\xi$ at $\xi=0$, and exponentially decaying as $\xi\rightarrow+\infty$ on one hand,
 and solutions $w$ of the transformed equation (\ref{vm}) that are Lipschitz-class, with positive $\psi$-direction derivative at $\psi=0$, and vanishing quadratically at $\psi=\psi^0(y)$ on the other. Additionally, the boundedness assumption on $\p_{\xi \xi} v$ means that $\sqrt{w} \p_{\psi\psi } w$ is bounded.

 Hence, if there are two solutions to the Prandtl equation satisfying the assumptions of Corollary \ref{cor:uniqueness}, they yield two solutions, in the sense of \propref{prop:uniqueness-vm}, of (\ref{vm}) by the von Mises transform, and these have just been shown to be equal.
 This ends the proof of Corollary \ref{cor:uniqueness} and thereby of Theorem \ref{exist}.
\end{proof}

\subsection*{Acknowledgements.}

This project has received funding from the European Research Council (ERC) under the European Union's Horizon 2020 research and innovation program Grant agreement No 637653, project BLOC ``Mathematical Study of Boundary Layers in Oceanic Motion''. The authors have also been partially funded by the ANR project Dyficolti ANR-13-BS01-0003-01.

\begin{small}
\bibliography{prbib}

\begin{thebibliography}{10}

\bibitem{BCT}
V.~Barcilon, P.~Constantin, and E.~S. Titi.
\newblock {Existence of solutions to the {S}tommel-{C}harney model of the
  {G}ulf {S}tream}.
\newblock {\em SIAM J. Math. Anal.}, 19(6):1355--1364, 1988.

\bibitem{BC}
D.~Bresch and T.~Colin.
\newblock {Some remarks on the derivation of the {S}verdrup relation}.
\newblock {\em J. Math. Fluid Mech.}, 4(2):95--108, 2002.

\bibitem{BGV}
D.~Bresch and D.~G{\'e}rard-Varet.
\newblock {Roughness-induced effects on the quasi-geostrophic model}.
\newblock {\em Comm. Math. Phys.}, 253(1):81--119, 2005.

\bibitem{DSR}
A.~Dalibard and L.~Saint-Raymond.
\newblock {Mathematical study of degenerate boundary layers}.
\newblock preprint hal-00682477v2, to be published in Memoirs of the AMS, 2015.

\bibitem{DG}
B.~Desjardins and E.~Grenier.
\newblock {On the homogeneous model of wind-driven ocean circulation}.
\newblock {\em SIAM J. Appl. Math.}, 60(1):43--60, 2000.

\bibitem{Fabook}
A.~Friedman.
\newblock {\em {Partial differential equations of parabolic type}}.
\newblock Prentice-Hall, Inc., Englewood Cliffs, N.J., 1964.

\bibitem{IerleyRuehr}
G.~R. Ierley and O.~G. Ruehr.
\newblock {Analytic and numerical solutions of a nonlinear boundary-layer
  problem}.
\newblock {\em Stud. Appl. Math.}, 75(1):1--36, 1986.

\bibitem{Lpy}
P.-Y. Lagr{\'e}e.
\newblock {Interactive boundary layer ({IBL})}.
\newblock In {\em {Asymptotic methods in fluid mechanics: survey and recent
  advances}}, volume 523 of {\em {CISM Courses and Lectures}}, pages 247--286.
  SpringerWienNewYork, Vienna, 2010.

\bibitem{Oo63}
O.~A. Ole{\u\i}nik.
\newblock {On the system of {P}randtl equations in boundary-layer theory}.
\newblock {\em Dokl. Akad. Nauk SSSR}, 150:28--31, 1963.

\bibitem{OSbook}
O.~A. Oleinik and V.~N. Samokhin.
\newblock {\em {Mathematical models in boundary layer theory}}, volume~15 of
  {\em {Applied Mathematics and Mathematical Computation}}.
\newblock Chapman \& Hall/CRC, Boca Raton, FL, 1999.

\bibitem{Pjbook}
J.~Pedlosky.
\newblock {\em {Ocean Circulation Theory}}.
\newblock Springer, 1996.

\bibitem{SmithMallier}
C.~J. Smith and R.~Mallier.
\newblock {Removing the separation singularity in a barotropic ocean with
  bottom friction}.
\newblock In {\em {Proceedings of the {I}nternational {C}onference on
  {B}oundary and {I}nterior {L}ayers---{C}omputational and {A}symptotic
  {M}ethods ({BAIL} 2002)}}, volume 166, pages 281--290, 2004.

\end{thebibliography}
\bibliographystyle{abbrv}
\end{small}

\end{document}